\documentclass[12pt]{article}
\usepackage{amsmath,amsthm,amssymb,amsfonts}
\usepackage{enumerate}
\usepackage{xcolor}
\usepackage{authblk}
\usepackage{appendix}
\usepackage{hyperref}

\usepackage[left=1in,right=1in,top=1in,bottom=1in]{geometry}

\usepackage{graphicx}
\usepackage{subfigure}
\usepackage{booktabs}
\usepackage{multirow}
\usepackage{xcolor}
\usepackage{kvsetkeys}

\theoremstyle{plain}
\newtheorem{theorem}{\hskip\parindent Theorem}
\newtheorem{lemma}{\hskip\parindent Lemma}
\newtheorem{corollary}{\hskip\parindent Corollary}

\theoremstyle{definition}
\newtheorem{remark}{\hskip\parindent Remark}

\numberwithin{equation}{section}

\newcommand\Ai{\operatorname{Ai}}
\newcommand\Bi{\operatorname{Bi}}
\renewcommand\Re{\operatorname{Re}}
\renewcommand\Im{\operatorname{Im}}
\newcommand\sgn{\operatorname{sgn}}

\begin{document}


\title{Classification of the real Painlev\'{e} I transcendents
 by zeros and connection problem: an asymptotic study}

\author[$\dag$]{Yan Huang}

\author[$\ddag$]{Yu-Tian Li}

\author[$\dag$]{Wen-Gao Long \thanks{Corresponding author: longwg@hnust.edu.cn}}

\affil[$\dag$]{\small School of Mathematics and Statistics,
Hunan University of Science and Technology, Xiangtan, Hunan, 411201, PR China}

\affil[$\ddag$]{\small
School of Science and Engineering, The Chinese University of Hong Kong, Shenzhen, Guangdong, 518172, PR China.}

\date{}

\maketitle

\begin{abstract}

In this paper, we study the asymptotic behavior and connection problem of  Painlev\'e I (PI) equation through a detailed analysis of the Stokes multipliers associated with its solutions. Focusing on the regime where the derivative at the real zeros of the solution becomes large, we apply the complex WKB method to derive full asymptotic expansions of the Stokes multipliers. These expansions allow us to classify real solutions of PI according to their behavior at the zeros, distinguishing between oscillatory, separatrix, and singular types solutions on the negative real axis. Furthermore, we resolve the connection problem between the large negative asymptotics and the location of positive zeros by establishing full asymptotic expansions of the zero parameters. Our approach enables the construction of a precise phase diagram in the
$(r,b)$-plane, where $r$ is the location of a zero and $b$ is the derivative at that point. Numerical simulations are provided to validate the theoretical results. This work extends prior studies on monodromy asymptotics and contributes a comprehensive framework for understanding the global structure of real PI solutions through their local zero data.

%

\end{abstract}

\vspace{5mm}

\vspace{5mm}

\section{Introduction}
The six Painlev\'e equations (PI–PVI) originated from early 20th-century efforts by Paul Painlev\'e and his collaborators to classify second-order ordinary differential equations whose solutions are free from movable singularities other than poles -- a criterion now known as the \emph{Painlev\'e property}. This property ensures that the solutions define new transcendental functions, known as \emph{Painlev\'e transcendents}, which cannot be expressed in terms of classical special functions.
Each equation in this family exhibits distinctive characteristics and arises in a variety of physical contexts, ranging from statistical mechanics to quantum gravity~\cite{Ablowitz03,BMW1973,EFIK1996,JMMS1980,Kan2002,McCoy1992, TW2011,WMTB1976}.

Among them, the first Painlev\'e equation (PI),
\begin{equation}
\label{PI equation}
\frac{d^{2}y}{dt^{2}} = 6y^2 + t,
\end{equation}
is notable for its simplicity and the rich structure of its solutions.
The PI equation exhibits intricate asymptotic behaviors and connection problems, which have been the subject of extensive mathematical investigation. The pioneering asymptotic analyses were conducted by Joshi and Kruskal~\cite{Joshi-Kruskal-1992} and Kapaev and Kitaev~\cite{Kapaev-Kitaev-1993}. As $t \to \infty$, solutions of PI display diverse asymptotic expansions -- including oscillatory and pole-free behaviors -- depending on the sector in the complex plane. Understanding these behaviors is essential for uncovering the global properties of PI solutions.

A central aspect of the analysis of PI involves the connection problem: relating local properties of the solution (such as initial data or pole distribution) to its global asymptotic behavior at infinity. As emphasized on multiple occasions by Clarkson~\cite{CPA2003,CPA2006}, the connection problem -- particularly in the context of the initial value problem -- remains an open question. More recent studies~\cite{Bender-Komijani-2015, Bender-Komijani-Wang, LongLi, LongLiWang} have continued to develop this area, although several key aspects remain unresolved.

The zeros of PI solutions are of particular interest in solving boundary value problems, as demonstrated in the work of Joshi and Kitaev~\cite{Joshi-Kitaev-2005}. In this paper, we aim to further investigate the connection problem by establishing the relationship between the behavior of solutions near their zeros and their asymptotic behavior as $t \to -\infty$, which we briefly outline below.

\medskip

Monodromy theory plays a central role in analyzing the asymptotic behavior and connection problems of Painlev\'e transcendents. For the first Painlev\'e equation (PI), there exists a one-to-one correspondence between its solutions and associated Stokes multipliers: a pair of Stokes multipliers uniquely determines a PI solution, and vice versa. Detailed information on the monodromy theory for PI and the definition of the Stokes multipliers is provided in Appendix~\ref{sec:AppA}. Kapaev~\cite{AAKapaev-1988} studied the asymptotic behavior of $y(t)$ as $t \to -\infty$ and classified the PI solutions into three distinct types:

\begin{enumerate}[(A)]
\item a two-parameter family of solutions, oscillating about the parabola $y=-\sqrt{-t/6}$ and satisfying
\begin{equation}\label{eq-behavior-type-A}
y=-\sqrt{-\frac{t}{6}}+d\,(-t)^{-\frac{1}{8}}\cos{\left[24^{\frac{1}{4}}\left(\frac{4}{5}(-t)^{\frac{5}{4}}-\frac{5}{8}d^2 \log(-t)+\varphi\right)\right]}+\mathcal{O}\left (t^{-\frac{5}{8}}\right )
\end{equation}
as $t\rightarrow-\infty$, where
    \begin{equation}\label{eq-parameter-d-theta}
  \left\{\begin{aligned}
  &24^{\frac{1}{4}}d^{2}=-\frac{1}{\pi}\log{|s_{0}|},\\
  &24^{\frac{1}{4}}\theta=\arg{s_{2}}-24^{\frac{1}{4}}d^{2}\left(\frac{19}{8}\log{2}+\frac{5}{8}\log{3}\right)-\frac{\pi}{4}-\arg\Gamma\left(-i\frac{24^{\frac{1}{4}}}{2}d^{2}\right);
\end{aligned}\right.
\end{equation}
\item a one-parameter family of solutions (termed {\it separatrix solutions}), satisfying
    \begin{equation}\label{eq-behavior-type-B}
      y(t)=y_{0}(t)-\frac{h}{4\sqrt{\pi}}24^{-\frac{1}{8}}(-t)^{-\frac{1}{8}}\exp\left\{-\frac{4}{5}24^{\frac{1}{4}}(-t)^{\frac{5}{4}}\right\}\left(1+\mathcal{O}\left(|t|^{-\frac{5}{4}}\right)\right)
    \end{equation}
    as $t\rightarrow-\infty$, where $y_{0}(t)=\sqrt{-t/6}\left[1+\mathcal{O}\left((-t)^{-\frac{5}{2}}\right)\right]$ and
    \begin{equation}\label{eq-h-s1-s-1}
    h=s_{1}-s_{-1};
    \end{equation}
\item a two-parameter family of solutions, having infinitely many double poles on the negative real axis and satisfying
    \begin{equation}\label{eq-behavior-type-C}
      \frac{1}{y(t)+\sqrt{{-t}/{6}}}\sim \frac{\sqrt{6}}{2}\sin^{2}\left\{\frac{2}{5}24^{1/4}(-t)^{\frac{5}{4}}+\frac{5}{8}\rho\log(-t)+\sigma\right\}
    \end{equation}
    as $t\rightarrow-\infty$, where
    \begin{equation}\label{eq-parameter-rho-sigma}
    \left\{\begin{aligned}
    \rho&=\frac{1}{2\pi}\log(|s_{2}|^{2}-1)=\frac{1}{2\pi}\log(|1+s_{2}s_{3}|)=\frac{1}{2\pi}\log|s_{0}|,\\
    \sigma&=\frac{19}{8}\rho\log{2}+\frac{5}{8}\rho\log{3}+\frac{1}{2}\arg\Gamma\left(\frac{1}{2}-i\rho\right)+\frac{\pi}{4}+\frac{1}{2}\arg{s_{2}}.
    \end{aligned}\right.
    \end{equation}
\end{enumerate}
In the above formulas, the quantities $s_k$ denote the Stokes multipliers, which will be explicitly defined later in Eq.~\eqref{eq-Stokes-matrices}. The classification of the three solution types is based on the following conditions:
\begin{equation}\label{eq-classifed-by-stokes}
\begin{cases}
\text{Type A (oscillatory solutions)}: & \Im s_{0}=1+s_{2}s_{3}>0;\\
\text{Type B (separatrix solutions)}: & \Im s_{0}=1+s_{2}s_{3}=0;\\
\text{Type C (singular solutions)}:  & \Im s_{0}=1+s_{2}s_{3}<0.
\end{cases}
\end{equation}
However, obtaining explicit expressions for the Stokes multipliers $s_k$ at fixed values of $t$ (e.g., $t = 0$) or as $t \to +\infty$ remains a significant challenge. This difficulty highlights the unresolved nature of the connection problem for the first Painlev\'e equation, as emphasized by Clarkson.

Rather than attempting to derive explicit expressions for the Stokes multipliers, Long and Li~\cite{LongLi} investigated their leading-order asymptotic behavior when either $y(0)$ or $y'(0)$ becomes large. They provided an asymptotic classification of the three types of PI solutions in terms of their initial data and derived certain limiting connection formulas. More recently, Long, Li, and Wang~\cite{LongLiWang} studied the connection problem for the PI equation between poles and large negative values of $t$, analyzing the asymptotic behavior of the Stokes multipliers as the pole parameters tend to infinity.

\medskip

In this paper, we investigate the connection problem for the first Painlev\'e equation (PI), focusing on the relationship between the three types of asymptotic behavior described in Eqs.~\eqref{eq-behavior-type-A}, \eqref{eq-behavior-type-B}, and \eqref{eq-behavior-type-C}, and the initial data near real zeros of the solution. Suppose $t = r$ is a real zero of a given PI solution. Then, the local behavior near $r$ is described by
\begin{equation}\label{eq-behavior-at-zero}
y(t) =b(t-r) + \frac{r}{2}(t-r)^2+\mathcal{O}((t-r)^{3}), \qquad\text{as}~ t \to r,
\end{equation}
where $b=y'(r)$. The pair $(r, b)$ uniquely determines the solution.

Our main objective is to compute the asymptotic expansion of the Stokes multipliers in the regime where at least one of the parameters $r$ or $b$ becomes large. This work can be seen as an extension of the results in~\cite{LongLi}, where the authors studied the asymptotics of the Stokes multipliers $s_k$ under the assumption that either $y(0)$ or $y'(0)$ is large. According to monodromy theory, the Stokes multipliers $s_k$ can, in general, be regarded as functions of the triple $(t, y(t), y'(t))$. Consequently, it is natural to ask how the Stokes multipliers behave under the condition $y(t) = 0$, particularly when either $t$ or $y'(t)$ is large.

One motivation for treating $s_k$ as functions of $r$ and $b$, and for studying their asymptotic behavior in this setting, is the possibility of classifying PI solutions in terms of the parameters $(r, b)$. Inspired by~\cite{LongLiWang}, we focus on the regime in which $b$ is large while $r$ remains arbitrary. Under this assumption, we are able to derive an asymptotic description of the phase diagram of PI solutions in the $(r, b)$-plane.

A second motivation arises from~\cite{Long-Jiang-Li}, in which the authors obtained large-$n$ asymptotics for the $n$-th pole of the real tritronqu\'ee solution. A natural parallel question concerns the behavior of the $n$-th zero of this particular solution as $n \to +\infty$. In this paper, we aim to derive analogous results for all real PI solutions, noting that relatively little is currently known about the distribution of their zeros on the positive real axis. Addressing this question requires a detailed asymptotic analysis of the Stokes multipliers in the limit where either $r$ or $b$ becomes large.

The remainder of this paper is organized as follows. In Section~\ref{sec:complex}, we apply the complex WKB method \cite{APC, Dunster-2014, Kawai-Takei-2005-book} to derive higher-order asymptotic expansions of the Lax pair as $b \to \pm\infty$. Section~\ref{sec:results} presents our main results, which are structured into three parts:
(1) full asymptotic expansions of the Stokes multipliers;
(2) classification of real PI solutions in terms of the zero parameters; and
(3) full asymptotic expansions for the locations of large zeros on the positive real axis.
This section also includes the proofs and numerical validations of the results in parts (2) and (3). In Section~\ref{sec:proof-theorem}, we provide the proof of part (1).

\section{Higher-order approximations of the Lax pair}\label{sec:complex}

Consider Eq.~\eqref{eq-system-hat-Phi} as a linear system of ODEs for $\widehat{\Phi} = (\phi_1, \phi_2)^T$, where $T$ denotes the transpose. Then, $\phi_{1}$ satisfies the following second-order ordinary differential equation:
\begin{equation}\label{Schrodinger-equation-t-general}
\frac{d^{2}\phi_{1}}{d\lambda^{2}}=\left[y_{t}^{2}+4\lambda^{3}+2\lambda t-2y t-4y^{3}-\frac{y_{t}}{\lambda-y}+\frac{3}{4}\frac{1}{(\lambda-y)^2}\right]\phi_{1}.
\end{equation}
Let $t\to r$, where $r$ is a zero of the PI solution. Then the above equation reduces to
\begin{equation}\label{Schrodinger-equation-reduced-t=a}
\frac{d^{2}\phi_{1}}{d\lambda^{2}}=\left[4\lambda^{3}+2\lambda r+b^2-\frac{b}{\lambda}+\frac{3}{4\lambda^2}\right]\phi_{1},
\end{equation}
where $b=y'(r)$.

Introduce a large parameter $\xi\to+\infty$ with
\begin{equation}\label{para-a-b-xi}
r=2A\xi^{\frac{4}{5}} \quad\text{and}\quad
b^2=4\xi^{\frac{6}{5}}\left[1+\sum\limits_{s=1}^{\infty}\frac{B_{s}}{\xi^{s}}\right],
\end{equation}
where $A, B_{s}\in\mathbb{R}$ are real and bounded for every $s\in\mathbb{N}$.
Under the scaling $\lambda=\xi^{2/5}z$ and $Y(z)=\phi(\xi^{2/5}z)$, Eq. \eqref{Schrodinger-equation-t-general} becomes
\begin{equation}\label{second-order-equation-w}
\frac{d^{2}Y}{dz^2}
=\xi^2 \left[f_{0}(z)+\sum\limits_{s=1}^{\infty}\frac{f_{s}(z)}{\xi^{s}}\right]Y
:=\xi^2 F(z,\xi) Y
\end{equation}
with
\begin{equation}
\begin{split}
f_{0}(z)=4(z^3+Az+1),\quad f_{1}(z)=-\frac{2\cdot\sgn(b)}{z}+4B_{1}, \quad
f_{2}(z)&=\frac{3}{4z^2}+4B_{2}-\frac{B_{1}}{z},
\end{split}
\end{equation}
and
\begin{equation}\label{eq-def-fs}
f_{s}(z)=4B_{s}-\frac{2g_{s}(\sgn(b), B_{1},\cdots,B_{s-1})}{z}, \quad s=3,4,\cdots,
\end{equation}
where $g_{s}$ is a polynomial in $\sgn(b), B_{1}, \ldots, B_{s-1}$, and $g_{s} = 0$ if all the coefficients $B_{j} = 0$ for $1 \leq j < s$.

Let $z_{i}$, $i = 0, 1, 2$, denote the roots of the cubic polynomial $z^{3} + Az + 1 = 0$, where $\Re z_{0} < 0$, and the other two roots $z_{1}, z_{2}$ are either both real with $\Re z_{1} \geq \Re z_{2}$, or form a complex conjugate pair with $\Im z_{1} \geq \Im z_{2}$. Define
\begin{equation}\label{eq-def-kappa2-hat-kappa2}
\kappa^2:=\kappa^{2}(A)=\frac{2}{\pi i}\int_{z_{1}}^{z_{2}}f_{0}(s)^{\frac{1}{2}}ds
\quad \text{and}\quad
\widehat{\kappa}^2:=\widehat{\kappa}^{2}(A)=\frac{2}{\pi i}\int_{z_{0}}^{z_{1}}f_{0}(s)^{\frac{1}{2}}ds.
\end{equation}
Here and throughout, the branch cuts are chosen such that $\arg(z - z_{0}) \in (-\pi, \pi)$ and $\arg(z - z_{i}) \in \left(-\frac{\pi}{2}, \frac{3\pi}{2}\right)$ for $i = 1, 2$.

According to \cite[Lemma 2.1]{LongLiWang}, we have,
\begin{equation}
\kappa^{2}\begin{cases}>0, \quad & A<-3/2^{\frac{2}{3}},\\
=0, \quad & A=-3/2^{\frac{2}{3}},\\
<0, \quad & A>-3/2^{\frac{2}{3}}.
\end{cases}
\end{equation}
Moreover, from the same lemma, there exists a constant $C_{0}\approx 2.004 860 503 264 124$ such that
\begin{equation}\label{Imkappa2}
\Im\widehat{\kappa}^2\begin{cases}>0, \quad & A<C_{0},\\
=0, \quad & A=C_{0},\\
<0, \quad & A>C_{0}.
\end{cases}
\end{equation}

When $A\neq -3/2^{\frac{2}{3}}$, define a transformation $\zeta:=\zeta(z)$ by
\begin{equation}
\label{def-zeta}
\frac{2}{3}\zeta^{\frac{3}{2}}=\int_{z_{1}}^{z}f_{0}(s)^{\frac{1}{2}}ds.
\end{equation}
Then $\zeta(z)$ defines a conformal mapping in a neighborhood of $z = z_{1}$ and along the two adjacent Stokes lines $\{\arg{\zeta} = \frac{\pi}{3}, \pi\}$ emanating from $z = z_{1}$ to infinity. Moreover, a direct computation of the integral in~\eqref{def-zeta} yields the asymptotic expansion
\begin{equation}\label{eq-asymp-relation-zeta-z}
\frac{2}{3}\zeta^{\frac{3}{2}}=\frac{4}{5}z^{\frac{5}{2}}+2Az^{\frac{1}{2}}+\alpha_{0}+\mathcal{O}(z^{-\frac{1}{2}})
\end{equation}
as $z \to \infty$, uniformly for $\arg z \in \left[0, \frac{4\pi}{5}\right]$, where
\begin{equation}\label{eq-def-alpha0}
\alpha_{0}=2\int_{z_{1}}^{\infty e^{i\theta}}\left[(s^3+As+1)^{\frac{1}{2}}-(s^{\frac{3}{2}}+\frac{A}{2}s^{-\frac{1}{2}})\right]ds-\left(\frac{4}{5}z_{1}^{\frac{5}{2}}+2Az_{1}^{\frac{1}{2}}\right)
\end{equation}
with $\theta \in \left[0, \frac{4\pi}{5}\right]$. Comparing~\eqref{eq-def-alpha0} with~\eqref{eq-def-kappa2-hat-kappa2}, we find that
\begin{equation}
\Re\alpha_{0}=\frac{\pi}{2}\Im\widehat{\kappa}^2,\qquad
\Im \alpha_{0}=\frac{\pi}{4}\kappa^2.
\end{equation}

Define
\[
W(\zeta,\xi)=\left(\frac{f_{0}(z)}{\zeta}\right)^{\frac{1}{4}}Y,
\]
then Eq.~\eqref{second-order-equation-w} transforms into
\begin{equation}\label{second-order-equation-W}
\frac{d^{2}W}{d\zeta^2}=\xi^2 \left[\zeta+\sum\limits_{s=1}^{\infty}\frac{\varphi_{s}(\zeta)}{\xi^{s}}\right]W,
\end{equation}
where
\begin{equation}\label{eq-def-varphis}
\begin{split}
\varphi_{1}(\zeta)&=\frac{\zeta f_{1}(z)}{f_{0}(z)},\\
\varphi_{2}(\zeta)&=\frac{5}{16\zeta^2}+\frac{\zeta[4f_{0}(z)f_{0}''(z)-5f_{0}'(z)^2]}{16f_{0}(z)^3}+\frac{\zeta f_{2}(z)}{f_{0}(z)},\\
\varphi_{s}(\zeta)&=\frac{\zeta f_{s}(z)}{f_{0}(z)}, \qquad s=3,4,\cdots
\end{split}
\end{equation}
are all analytic at $\zeta=0$.

Similar to \cite[Lemma 1]{Long-Jiang-Li}, we state the following result.
\begin{lemma}\label{lem-higher-order-approximation-ODE-1}
Let $\delta$ and $\epsilon$ be fixed small positive constants, and define
\[
\mathbb{S}:=\{\zeta\in\mathbb{C}~:~\arg{\zeta}=\frac{k\pi}{3},~ k=-1, 1,3\}.
\]
For any $n\in\mathbb{N}$, define
\begin{equation}\label{eq-def-An}
\widehat{\zeta}=\zeta+\mathcal{A}_{n}(\zeta,\xi),\quad\text{where}\quad
\mathcal{A}_{n}(\zeta,\xi)=\sum\limits_{s=1}^{2n}\frac{a_{s}(\zeta)}{\xi^{s}}
\end{equation}
with each $a_{s}(\zeta)$ analytic in a domain $\mathbb{D}$ containing $\mathbb{S}$. Suppose that
\begin{equation}\label{eq-An-coeff-compare}
\xi^2\left\{\mathcal{A}_{n}+(\zeta+\mathcal{A}_{n})(2+\mathcal{A}'_{n})\mathcal{A}'_{n}\right\}+\frac{3{\mathcal{A}''_{n}}^{2}-2(1+\mathcal{A}'_{n})\mathcal{A}'''_{n}}{4(1+\mathcal{A}'_{n})^2}=\xi^2\sum\limits_{s=1}^{2n}\frac{\varphi_{s}(\zeta)}{\xi^{s}}+\mathcal{O}(\xi^{-2n+1})
\end{equation}
as $\xi\to+\infty$, uniformly for all $\zeta\in\mathbb{D}$.
Then, for each $s\in\mathbb{N}$, the limits
\begin{equation}\label{eq-def-alpha-s}
\alpha_{s}:=\alpha_{s}(A,B_{1},\cdots, B_{s})=\lim\limits_{\zeta\to\infty}\zeta^{\frac{1}{2}}a_{s}(\zeta)
\end{equation}
exist and depend only on $A$ and $B_{1}, \ldots, B_{s}$.
Moreover, for any solution $W(\zeta, \xi)$ of Eq.~\eqref{second-order-equation-W},
there exist constants $C_{1}$ and $C_{2}$ such that
\begin{equation}\label{eq-higher-approximation-ODE-1}
W(\zeta,\xi)=\left[C_{1}+r_{1}(\zeta,\xi)\right]\Ai_{n}(\zeta,\xi)
+\left[C_{2}+r_{2}(\zeta,\xi)\right]\Bi_{n}(\zeta,\xi),
\end{equation}
where
\begin{equation}
\Ai_{n}(\zeta,\xi)=\left(\frac{d\widehat{\zeta}}{d\zeta}\right)^{-\frac{1}{2}}\Ai\left(\xi^{\frac{2}{3}}\widehat{\zeta}\right)\quad \text{ and }\quad \Bi_{n}(\zeta,\xi)=\left(\frac{d\widehat{\zeta}}{d\zeta}\right)^{-\frac{1}{2}}\Bi\left(\xi^{\frac{2}{3}}\widehat{\zeta}\right),
\end{equation}
and $r_{1,2}(\zeta,\xi)=\mathcal{O}(\xi^{-2n})$
as $\xi\to+\infty$, uniformly for all $\zeta$ in a neighborhood of $\mathbb{S}$.
\end{lemma}

\begin{remark}\label{rem-b=0-1}
The main difference between the above lemma and \cite[Lemma 1]{Long-Jiang-Li} lies in the definition of the functions $a_{s}(\zeta)$. Combining Eqs.~\eqref{eq-def-An} and~\eqref{eq-An-coeff-compare} with the analyticity of $a_{s}(\zeta)$, we observe that the functions $a_{s}(\zeta)$ are determined recursively:
\begin{equation}\label{eq-explicit-representation-a-s}
\begin{split}
a_{1}(\zeta) &= \frac{1}{2\zeta^{1/2}} \int_{0}^{\zeta} \frac{\varphi_{1}(t)}{t^{1/2}} \, dt, \\
a_{2}(\zeta) &= \frac{1}{2\zeta^{1/2}} \int_{0}^{\zeta} \frac{\varphi_{2}(t) - 2a_{1}(t)a_{1}'(t) - t(a_{1}'(t))^2}{t^{1/2}} \, dt, \\
a_{s}(\zeta) &= \frac{1}{2\zeta^{1/2}} \int_{0}^{\zeta} \frac{\varphi_{s}(t) + F_{s}(t)}{t^{1/2}} \, dt, \qquad s > 2,
\end{split}
\end{equation}
where $F_{s}(t)$ are polynomials in $a_{1}(t), a_{2}(t), \ldots, a_{s-1}(t)$ and their derivatives.

A direct calculation yields the explicit form of $\alpha_1$:
\begin{equation}
\label{eq-explicit-represent-a1pha1}
\alpha_{1} = \frac{1}{2} \int_{0}^{\infty} \frac{\varphi_{1}(t)}{t^{1/2}} \, dt = 2B_{1} \alpha_{1,1} + \alpha_{1,2},
\end{equation}
where
\begin{equation}\label{eq-def-alpha11-alpha12}
\alpha_{1,1} := \alpha_{1,1}(A) = \int_{z_{1}}^{\infty} \frac{1}{f_{0}(z)^{1/2}} \, dz, \quad
\alpha_{1,2} := \alpha_{1,2}(A, \sgn(b)) = -\int_{z_{1}}^{\infty} \frac{\sgn(b)}{z f_{0}(z)^{1/2}} \, dz.
\end{equation}

For $s \geq 2$, using the definitions of $\varphi_{s}(\zeta)$ and $f_{s}(z)$ in Eqs.~\eqref{eq-def-varphis} and~\eqref{eq-def-fs}, respectively, we can express $\alpha_{s}$ in the form
\begin{equation}\label{eq-facterization-alphas}
\alpha_{s} = 2B_{s} \alpha_{s,1} + \alpha_{s,2},
\end{equation}
where
\begin{equation}\label{eq-def-alphas1-alphas2}
\alpha_{s,1} := \alpha_{1,1}(A), \qquad
\alpha_{s,2} := \alpha_{s,2}(A, \sgn(b), B_{1}, \ldots, B_{s-1}).
\end{equation}
\end{remark}

\begin{remark}
Although the definition of $a_{s}(\zeta)$ differs from that in \cite[Lemma 1]{Long-Jiang-Li}, the proof of the lemma above can still be carried out following the analysis in \cite[Section 3]{Long-Jiang-Li}, based on the key fact that $\varphi_{s}(\zeta) = \mathcal{O}(\zeta^{-4/5})$ as $\zeta \to \infty$. Specifically, this asymptotic behavior allows us to show by induction that
\begin{equation}
a_{s}(\zeta) = \mathcal{O}(\zeta^{-1/2}) \quad \text{and} \quad a_{s}'(\zeta) = \mathcal{O}(\zeta^{-3/2}), \quad \text{as } \zeta \to \infty.
\end{equation}
The remaining details of the proof proceed in the same manner as those in \cite[Section 3]{Long-Jiang-Li}.
\end{remark}

\begin{remark}
It should be noted that the parametrization given in \eqref{para-a-b-xi} is not unique. For example, one may instead set all $B_{s} = 0$ and parametrize $r$ and $b$ as
\begin{equation}\label{para-a-b-xi-2}
r = A\xi^{\frac45}\quad\text{and}\quad b^2 = 4\xi^{\frac65},
\end{equation}
with $A$ being an arbitrary constant. Under this assumption, Lemma~\ref{lem-higher-order-approximation-ODE-1} still holds, provided we take $B_{s} = 0$ for all $s \in \mathbb{N}^{+}$ in the definition of $a_{s}(\zeta)$. In this case, the coefficients $\alpha_{s}$ depend only on $A$ and $\sgn(b)$.
\end{remark}

\section{Main results}
\label{sec:results}

\subsection{Full asymptotic expansions of the Stokes multipliers}

Using Lemma~\ref{lem-higher-order-approximation-ODE-1} and the Stokes phenomenon for Airy functions, we derive the full asymptotic expansions of the Stokes multipliers corresponding to PI solutions satisfying the conditions $y(r) = 0$ and $y'(r) = b$ as $b \to \pm\infty$. The results are formulated in four theorems, corresponding to the following cases:
\begin{equation}
\text{(i) } A<-3/2^{\frac{2}{3}}, \quad
\text{(ii) } -3/2^{\frac{2}{3}}<A<C_{0},\quad
\text{(iii) } A>C_{0}, \quad
\text{(iv) } A=C_{0},
\end{equation}
where the constant $C_{0}$ is defined in Eq.~\eqref{Imkappa2}.

This classification reflects the differences in the limiting Stokes geometries of Eq.~\eqref{second-order-equation-w} as $\xi \to +\infty$, which are illustrated in Figure~\ref{Figure-stokes-geometry}. In particular, the sign of $\Re \alpha_{0}$ varies as $A$ crosses the boundaries of these intervals, which in turn leads to qualitatively different asymptotic behaviors of the Stokes multipliers.

\begin{figure}[!ht]
\centering
\subfigure{
\includegraphics[width=0.43\textwidth]{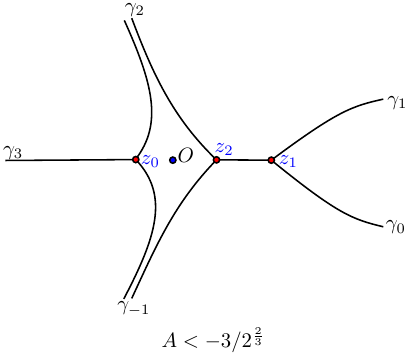}}\qquad
\subfigure{
\includegraphics[width=0.4\textwidth]{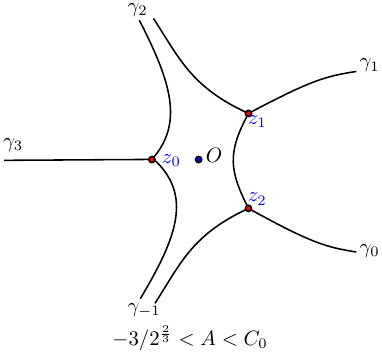}}\\
\subfigure{
\quad \includegraphics[width=0.35\textwidth]{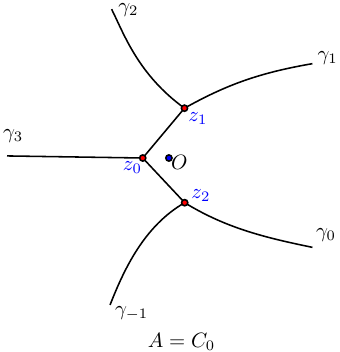}}\quad\qquad
\subfigure{
\includegraphics[width=0.4\textwidth]{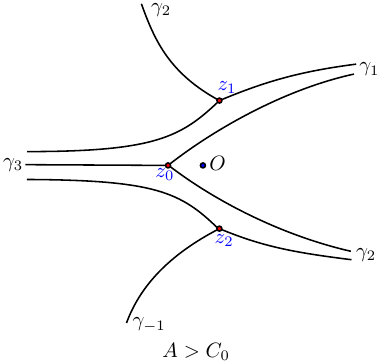}}
\caption{Stokes geometry defined by $\Re\int_{z_{i}}^{z}f_{0}(s)^{\frac{1}{2}}ds=0$ for different cases of the parameter $A$.}
\label{Figure-stokes-geometry}
\end{figure}

In cases (i), (ii), and (iii), as the parameter $A$ varies within the respective intervals, we may assume $B_{s} = 0$ for all $s = 1, 2, \ldots$, and adopt the parametrization given in Eq.~\eqref{para-a-b-xi-2}.
The asymptotic expansions of the Stokes multipliers in these three cases are stated in the following three theorems.

\begin{theorem}\label{Thm-stokes-case-I}
Let $r = 2A\xi^{2/5}$ and $b^2 = 4\xi^{6/5}$, where $A < -3/2^{2/3}$ and is bounded away from $-3/2^{2/3}$ for sufficiently large $\xi$. Then $\Re \alpha_{0} > 0$, and as $\xi \to +\infty$, the Stokes multiplier $s_{0}$ corresponding to the PI solution $y(t)$ satisfying $y(r) = 0$ and $y'(r) = b$ admits the asymptotic expansion
\begin{equation}
\label{eq-s0-case-I}
s_{0} \sim i \exp\left\{-\sum_{s=0}^{\infty} \frac{2\alpha_{s}}{\xi^{s-1}}\right\},
\end{equation}
where $\alpha_{0}$ is defined in \eqref{eq-def-alpha0}, and $\alpha_{s} := \alpha_{s}(A, \sgn(b))$ for $s \in \mathbb{N}$ are as given in Lemma~\ref{lem-higher-order-approximation-ODE-1}.
\end{theorem}

\begin{theorem}\label{Thm-stokes-case-II}
Let $r = 2A\xi^{2/5}$ and $b^2 = 4\xi^{6/5}$, where $-3/2^{2/3} < A < C_{0}$, and $A$ is bounded away from both $-3/2^{2/3}$ and $C_{0}$ for sufficiently large $\xi$. Then $\Re \alpha_{0} > 0$, and as $\xi \to +\infty$, the Stokes multipliers corresponding to the PI solution $y(t)$ with $y(r) = 0$ and $y'(r) = b$ admit the following asymptotic expansions:
\begin{equation}\label{eq-stokes-full-asymptotic-case-II}
\begin{aligned}
s_{0} &\sim 2i \exp\left\{ -\sum_{s=0}^{\infty} \frac{2\Re \alpha_{s}}{\xi^{s-1}} \right\}
\cos\left( \sum_{s=0}^{\infty} \frac{2\Im \alpha_{s}}{\xi^{s-1}} \right), \\
s_{1} &= -\overline{s_{-1}} \sim i \exp\left\{ \sum_{s=0}^{\infty} \frac{2\alpha_{s}}{\xi^{s-1}} \right\}, \\
s_{2} &= -\overline{s_{3}} \sim -i \exp\left\{ \sum_{s=0}^{\infty} \frac{-4i\Im \alpha_{s}}{\xi^{s-1}} \right\}.
\end{aligned}
\end{equation}
Here, the quantities $\alpha_{s} := \alpha_{s}(A, \sgn(b))$ are defined as in Lemma~\ref{lem-higher-order-approximation-ODE-1}.
\end{theorem}

\begin{theorem}\label{Thm-stokes-case-III}
Let $r = 2A\xi^{2/5}$ and $b^2 = 4\xi^{6/5}$, where $A > C_{0}$ and is bounded away from $C_{0}$ for sufficiently large $\xi$. Then $\Re \alpha_{0} < 0$, and as $\xi \to +\infty$, the Stokes multipliers corresponding to the PI solution $y(t)$ with $y(r) = 0$ and $y'(r) = b$ admit the following asymptotic expansions:
\begin{equation}\label{eq-stokes-full-asymptotic-case-III}
\begin{aligned}
s_{0}&\sim -i\exp\left\{-\sum\limits_{s=0}^{\infty}\frac{4\Re\alpha_{s}}{\xi^{s-1}}\right\},\\
s_{1}&=-\overline{s_{-1}}\sim i\exp\left\{\sum\limits_{s=0}^{\infty}\frac{2\alpha_{s}}{\xi^{s-1}}\right\},\\
s_{2}&=-\overline{s_{3}}\sim i\exp\left\{\sum\limits_{s=0}^{\infty}\frac{-2\alpha_{s}}{\xi^{s-1}}\right\}.
\end{aligned}
\end{equation}
Here, the quantities $\alpha_{s} := \alpha_{s}(A, \sgn(b))$ are defined as in Lemma~\ref{lem-higher-order-approximation-ODE-1}.
\end{theorem}

From the previous three theorems, we observe that \( s_{0} \to 0 \) or \( s_{0} \to \infty \) as \( \xi \to +\infty \). Therefore, these results do not allow us to resolve the connection problem for certain PI solutions between \( t \to -\infty \) and large positive zeros. To address this issue, we must fix the Stokes multipliers \( s_{k} \), which can be achieved by setting \( A = C_{0} \).
However, when \( A = C_{0} \) is fixed, the characterization of the Stokes multipliers for all real PI solutions cannot be accomplished under the assumption that \( B_{s} = 0 \) for all \( s \in \mathbb{N}\). Hence, we must adopt the general parametrization given in Eq.~\eqref{para-a-b-xi} with \( A = C_{0} \).
The result is stated as follows.

\begin{theorem}\label{Thm-stokes-case-IV}
Let
\[
r = 2C_{0}\xi^{2/5} \quad\text{and}\quad
b^2 = 4\xi^{6/5} \sum_{s=1}^{\infty} \frac{B_{s}}{\xi^{s}},
\]
with \( C_{0} \) as defined in Eq.~\eqref{Imkappa2}.
Then \( \Re \alpha_{0} = 0 \), and as \( \xi \to +\infty \), the Stokes multipliers corresponding to the PI solution satisfying \( y(r) = 0 \), \( y'(r) = b \) admit the following full asymptotic expansions:
\begin{equation}
\label{eq-stokes-full-asymptotic-case-IV}
\begin{split}
s_{0}&\sim i\left\{2\exp\left\{-\sum\limits_{s=0}^{\infty}\frac{2\Re\alpha_{s+1}}{\xi^{s}}\right\}\cos{\left[\sum\limits_{s=-1}^{\infty}\frac{2\Im\alpha_{s+1}}{\xi^{s}}\right]}-\exp\left\{-\sum\limits_{s=0}^{\infty}\frac{4\Re\alpha_{s+1}}{\xi^{s}}\right\}\right\},\\
s_{1}&=-\overline{s_{-1}}\sim i\exp\left\{\sum\limits_{s=0}^{\infty}\frac{2\Re\alpha_{s+1}}{\xi^{s}}+\sum\limits_{s=-1}^{\infty}\frac{2i\Im\alpha_{s+1}}{\xi^{s}}\right\},\\
s_{2}&=-\overline{s_{-3}}\sim -i\exp\left\{-\sum\limits_{s=-1}^{\infty}\frac{4i\Im\alpha_{s+1}}{\xi^{s}}\right\}+i\exp\left\{-\sum\limits_{s=0}^{\infty}\frac{2\Re\alpha_{s+1}}{\xi^{s}}-\sum\limits_{s=-1}^{\infty}\frac{2i\Im\alpha_{s+1}}{\xi^{s}}\right\}
\end{split}
\end{equation}
as $\xi\to+\infty$. Here, $\alpha_{s}$, $s=1,2,\dots$, are defined as in Lemma~\ref{lem-higher-order-approximation-ODE-1} with $A=C_{0}$.
\end{theorem}

\begin{remark}
As \( A \to -3/2^{2/3} \), the two turning points \( z_{1} \) and \( z_{2} \) coalesce, and Lemma~\ref{lem-higher-order-approximation-ODE-1} no longer applies in this limit. In such cases, a similar analysis can be carried out using parabolic cylinder functions. However, the four theorems presented above are sufficient for achieving the goals of this paper.
Specifically, by applying Theorems~\ref{Thm-stokes-case-I}, \ref{Thm-stokes-case-II}, and \ref{Thm-stokes-case-III}, we obtain a classification of PI solutions with respect to the initial data \( (r, b) \) at their zeros and can asymptotically construct the phase diagram of PI solutions in the \( (r, b) \)-plane. Furthermore, Theorem~\ref{Thm-stokes-case-IV} provides an asymptotic approximation for the locations of large positive zeros of certain real PI solutions.
\end{remark}

\subsection{\texorpdfstring{Classification of the PI solutions with respect to $r$ and $b$}{Classification of the PI solutions with respect to r and b}}

In this subsection, we establish the relationship between the zero parameters of real PI solutions and the three types of asymptotic behaviors described in Eqs.~\eqref{eq-behavior-type-A}, \eqref{eq-behavior-type-B}, and \eqref{eq-behavior-type-C}. Let \( r = 2A\xi^{4/5} \) and \( b^2 = 4\xi^{6/5} \). From Theorem~\ref{Thm-stokes-case-II}, we obtain the following result.

\begin{corollary}\label{cor-classification-1}
For any fixed \( A \in (-3/2^{2/3}, C_{0}) \), there exists an increasing sequence \( \{\xi_{n}\}_{n=0}^\infty \) such that:
\begin{enumerate}[(i)]
\item The PI solutions satisfying \( y(r_n) = 0 \), \( y'(r_n) = b_n \), where \( r_n = 2A\xi_n^{4/5} \) and \( b_n = \pm 2\xi_n^{3/5} \), are of type (B);
\item For \( \xi_{2m-1} < \xi < \xi_{2m} \), \( m = 1, 2, \ldots \), the PI solutions with \( y(r) = 0 \), \( y'(r) = b \) are of type (A);
\item For \( \xi_{2m} < \xi < \xi_{2m+1} \), \( m = 0, 1, 2, \ldots \), the PI solutions with \( y(r) = 0 \), \( y'(r) = b \) are of type (C).
\end{enumerate}

Moreover, there exists \( n_{0} \in \mathbb{Z} \) such that, as \( n \to \infty \), the sequence \( \xi_{n} \) satisfies the asymptotic expansion:
\begin{equation}\label{eq-asym-xi-n}
\xi_{n} \sim \frac{\left(n + \frac{1}{2} + n_{0}\right)\pi + 2\Im \alpha_{1}}{-2\Im \alpha_{0}} \left[ 1 + \sum_{s=2}^{\infty} \frac{\mu_{s}}{\left(n - \frac{1}{2} + n_{0} + \frac{2\Im \alpha_{1}}{\pi} \right)^{s}} \right],
\end{equation}
where the coefficients \( \mu_{s} \) can be expressed in terms of \( \Im \alpha_{j} \) for \( j = 0, 1, \ldots, s \). In particular, \( \mu_{1} = -\dfrac{4\Im \alpha_{0} \Im \alpha_{2}}{\pi^{2}} \).
\end{corollary}

\begin{proof}
Fix any \( A \in (-3/2^{2/3}, C_{0}) \). From the asymptotic expression for \( s_{0} \) in Eq.~\eqref{eq-stokes-full-asymptotic-case-II}, there exists a constant \( M > 0 \) such that for all \( \xi > M \), the function \( s_{0}(\xi) \) admits a sequence of simple zeros \( \{ \xi_n \} \), viewed as a function of \( \xi \). By choosing \( \xi_1 \) appropriately, we can also arrange that \( \Im s_{0}(\xi) > 0 \) for \( \xi_{2m-1} < \xi < \xi_{2m} \) and \( \Im s_{0}(\xi) < 0 \) for \( \xi_{2m} < \xi < \xi_{2m+1} \). Comparing this behavior with the classification criterion in Eq.~\eqref{eq-classifed-by-stokes}, we establish the assertions of the corollary, except for the asymptotic formula~\eqref{eq-asym-xi-n}.

To derive the asymptotics of \( \xi_n \), we return to Eq.~\eqref{eq-stokes-full-asymptotic-case-II}. Since \( \Im \alpha_{0} < 0 \), there exists \( n_{0} \in \mathbb{Z} \) such that
\begin{equation}\label{eq-s0-cos-compare}
-2\Im \alpha_{0}\cdot\xi_{n}-2\Im \alpha_{1}-\frac{2\Im \alpha_{2}}{\xi_{n}}-\cdots\sim \left(n+n_{0}+\frac{1}{2}\right)\pi
\end{equation}
as \( n \to \infty \). Solving this asymptotic relation for \( \xi_n \), we obtain
the desired asymptotic approximation of $\xi_n$ in~\eqref{eq-asym-xi-n},
where the coefficients \( \mu_s \) can be expressed in terms of \( \Im \alpha_{j} \), for \( j = 0, 1, \ldots, s \). In particular, by substituting the expansion~\eqref{eq-asym-xi-n} into~\eqref{eq-s0-cos-compare}, we find
\[
\mu_{2}=-\frac{4\Im\alpha_{0}\Im\alpha_{2}}{\pi^2}. \qedhere
\]
\end{proof}

The above corollary provides an asymptotic classification of real PI solutions in the limit \( b \to \infty \), for arbitrary \( r \in \mathbb{R} \). Since \( \Im \alpha_{0} = \frac{\pi \kappa^2}{4} \) depends on \( A \), the quantities \( \xi_{n} \), \( r_{n} \), and \( \hat{b}_{n} \) are all functions of \( A \). Therefore, Corollary~\ref{cor-classification-1} can be equivalently reformulated as follows: there exist two sequences of curves, denoted \( \Sigma_{n}^{+} \) and \( \Sigma_{n}^{-} \), such that the PI solution satisfying \( y(r) = 0 \), \( y'(r) = b \) belongs to type (B) when the point \( (r, b) \) lies on one of these curves.
Moreover, there exists a constant \( n_{0} \geq 0 \) such that, as \( n \to \infty \), the curves \( \Sigma_{n}^{\pm} \) admit the asymptotic descriptions:
\begin{equation}\label{eq-Sigma-n}
\Sigma_{n}^{\pm}:
\left\{\begin{aligned}
r^{\pm}:=r_{n}^{\pm}(A)\sim 2A\left[\frac{(n+\frac{1}{2}+n_{0})\pi\pm 2\Im \alpha_{1}}{-2\Im \alpha_{0}}\right]^{\frac{4}{5}},\\
b^{\pm}:=\hat{b}_{n}^{\pm}(A)\sim \pm 2 \left[\frac{(n+\frac{1}{2}+n_{0})\pi\pm 2\Im \alpha_{1}}{-2\Im \alpha_{0}}\right]^{\frac{3}{5}},
\end{aligned}\right.
\qquad \text{for }-3/2^{\frac{2}{3}}<A<C_{0}
\end{equation}
as $n\to\infty$.
These asymptotic curves are confirmed by numerical simulations; see Figure~\ref{fig-classification-curves}. Each blue curve terminates at an endpoint due to the restriction \( A \in (-3/2^{2/3}, C_{0}) \). As \( A \to -3/2^{2/3} \), the curves tend to infinity, while as \( A \to C_{0} \), they approach finite limiting points denoted by blue circles.

\begin{figure}[!ht]
\centering
\includegraphics[width=0.8\textwidth]{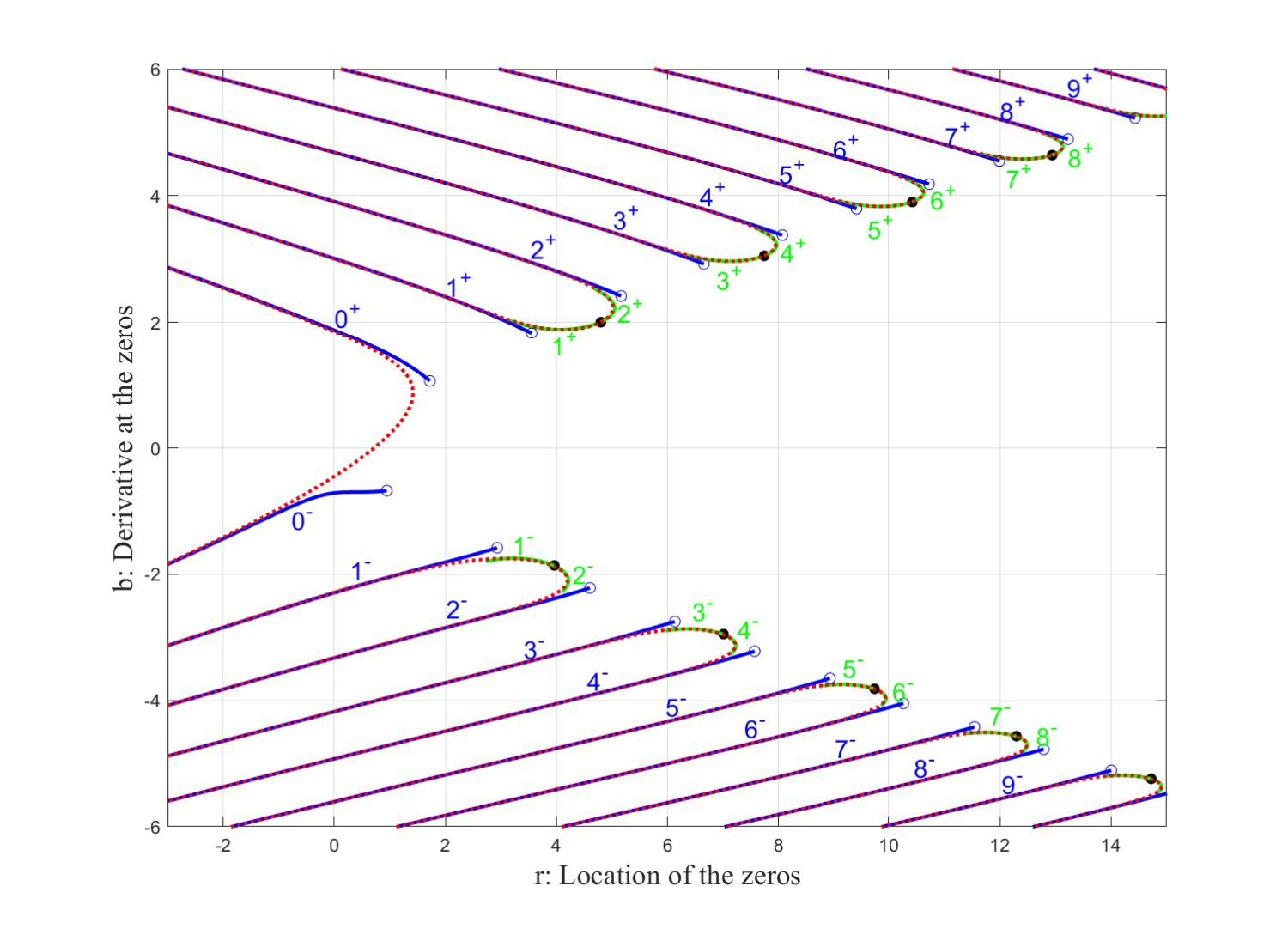}
\caption{
Asymptotic and numerical descriptions of curves \( \Sigma_{n}^{\pm} \) corresponding to separatrix solutions (type (B)).
The blue curves (labeled \( n^{\pm} \)) represent the asymptotic approximations of \( \Sigma_{n}^{\pm} \) given in Eq.~\eqref{eq-Sigma-n}, while the red dash-dotted curves represent numerically computed values of \( \Sigma_{n}^{\pm} \).
Each blue curve terminates at a blue circle due to the restriction \( A \in (-3/2^{2/3}, C_{0}) \) in the asymptotic expansion. The green curves show the asymptotic description of $\Sigma_{n}^{\pm}$ when $A$ is close to $C_{0}$; see Eq.~\eqref{eq-Sigma-n-extend}. All the green curves originate at a black point because $B_{1}\geq \Lambda_{0}$, where $\Lambda_{0}$ is defined in \eqref{eq-B1-restriction-def-Lambda0}.
}
\label{fig-classification-curves}
\end{figure}

\begin{remark}
When \( A \) is close to \( C_{0} \), the asymptotic description of the curves \( \Sigma_{n} \) given in Eq.~\eqref{eq-Sigma-n} does not match the numerical simulations uniformly; see the ``fingertips'' in Figure~\ref{fig-classification-curves}. This discrepancy arises because the asymptotic behavior of \( s_{0} \) described in Theorem~\ref{Thm-stokes-case-II} is not uniform in the joint limit \( \xi \to +\infty \) and \( A \to C_{0} \).

To capture the asymptotic behavior near the fingertips as \( n \to \infty \), we must instead use the asymptotic expansion of \( s_{0} \) from Theorem~\ref{Thm-stokes-case-IV}. From Eq.~\eqref{eq-stokes-full-asymptotic-case-IV}, we find that \( s_{0} = 0 \) implies
\begin{equation}
2\cos\left[2\xi\Im\alpha_{0}+2\sgn(b)\Im\alpha_{1}\right]=\exp\left[-2\Re\alpha_{1}\right]+\mathcal{O}(\xi^{-1})
\end{equation}
with $\alpha_{0}=\alpha_{0}(C_{0})$ and $\alpha_{1}=\alpha_{1}(C_{0})$.
Since $\Im\alpha_{0}=\frac{\pi\kappa^2}{4}<0$, then there exists a sequence $\{\widetilde{\xi}_{n}\}$ such that  $s_{0}(\widetilde{\xi}_{n})=0$ and
\begin{equation}\label{eq-approx-tilde-xi-n}
\begin{split}
\widetilde{\xi}_{2m-1}^{\pm}&=\frac{1}{2\Im\alpha_{0}}\left(\arccos\left\{\frac{1}{2}\exp[-2\Re\alpha_{1}]\right\}\mp 2\Im\alpha_{1}-2(m+m_{0})\pi\right)+\mathcal{O}(m^{-1}),\\
\widetilde{\xi}_{2m}^{\pm}&=\frac{1}{2\Im\alpha_{0}}\left(-\arccos\left\{\frac{1}{2}\exp[-2\Re\alpha_{1}]\right\}\mp 2\Im\alpha_{1}-2(m+m_{0})\pi\right)+\mathcal{O}(m^{-1})
\end{split}
\end{equation}
as $m\to\infty$.
Hence, the asymptotic description of the curves \( \Sigma_{2m-1} \) and \( \Sigma_{2m} \) can be extended near the fingertips as
\begin{equation}\label{eq-Sigma-n-extend}
\Sigma_{n}^{\pm}:
\left\{\begin{aligned}
&r^{\pm}:=\widetilde{r}_{n}^{\pm}(B_{1})\sim 2C_{0}\left(\widetilde{\xi}_{n}^{\pm}\right)^{\frac{4}{5}},\\
&b^{\pm}:=\widetilde{b}_{n}^{\pm}(B_{1})\sim \pm 2\left(\widetilde{\xi}_{n}^{\pm}\right)^{\frac{3}{5}}\left(1+\frac{B_{1}}{2\widetilde{\xi}_{n}^{\pm}}\right)
\end{aligned}\right.
\end{equation}
as $n\to\infty$. Since $\xi_{n}^{\pm}$ must be real, we conclude from \eqref{eq-explicit-represent-a1pha1} that
\begin{equation}
\frac{1}{2}\exp[-2\Re{\alpha_{1}}]=\frac{1}{2}\exp[-2(2B_{1}\Re{\alpha_{1,1}}+\Re{\alpha_{1,2}})]\leq 1,
\end{equation}
where $\alpha_{1,1}:=\alpha_{1,1}(C_{0}), \alpha_{1,2}:=\alpha_{1,2}(C_{0},\pm 1)$ are defined in \eqref{eq-def-alpha11-alpha12}. This implies that the parameter \( B_{1} \) must satisfy
\begin{equation}\label{eq-B1-restriction-def-Lambda0}
B_{1}\geq \Lambda_{0}^{\pm}:=\frac{-\log{2}-2\Re{\alpha_{1,2}(C_{0},\pm 1)}}{4\Re\alpha_{1,1}(C_{0})}
\end{equation}
in the approximation of $\widetilde{\xi}_{n}^{\pm}$ in \eqref{eq-approx-tilde-xi-n}.
Numerical simulations confirm that the asymptotic approximation in Eq.~\eqref{eq-Sigma-n-extend} is accurate when \( \Lambda_{0}^{\pm} \leq B_{1} \leq C \) for some fixed constant \( C > 0 \); see the green curves in Figure~\ref{fig-classification-curves}.
\end{remark}

From Corollary~\ref{cor-classification-1}, we also deduce that if the point \( (r, b) \) lies in the region between the curves \( \Sigma_{2m-1}^{+} \) and \( \Sigma_{2m}^{+} \) for some \( m \in \mathbb{N}^{+} \), then the corresponding PI solution is of type (A). Similarly, if \( (r, b) \) lies between \( \Sigma_{2m}^{+} \) and \( \Sigma_{2m+1}^{+} \) for some \( m \in \mathbb{N} \), then the solution is of type (C). An analogous classification holds for the regions between the curves \( \Sigma_{n}^{-} \).
These conclusions are corroborated by numerical simulations; see Figure~\ref{fig-classification-reigions}. It is also worth noting that oscillatory solutions exhibit two types of zeros: some are located in the pole region, while others are located in the oscillatory region. See Figure~\ref{fig-two-types-zeros} for illustrations of both types.

\begin{figure}[h]
\centering
\includegraphics[width=0.65\textwidth]{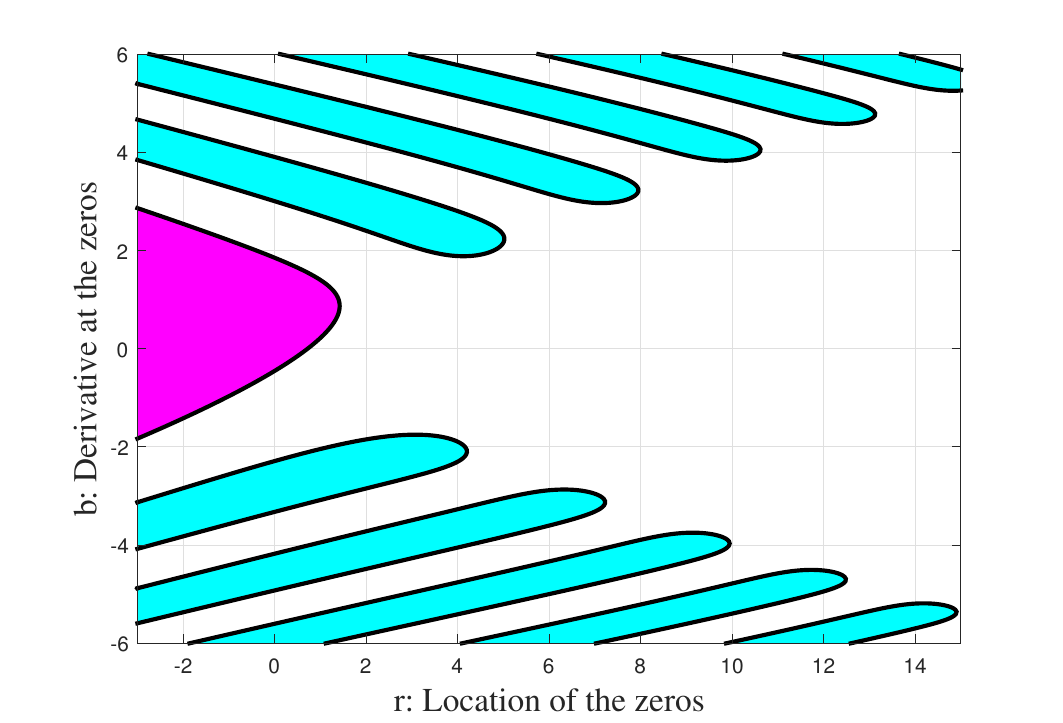}
\caption{
Phase diagram of real PI solutions in the \( (r, b) \)-plane.
The blue finger-like regions and the magenta region correspond to oscillatory solutions (type (A));
the boundaries of these regions represent separatrix solutions (type (B));
and the white regions correspond to singular solutions (type (C)).
}
\label{fig-classification-reigions}
\end{figure}

\begin{figure}[htp]
\centering
\includegraphics[width=0.68\textwidth]{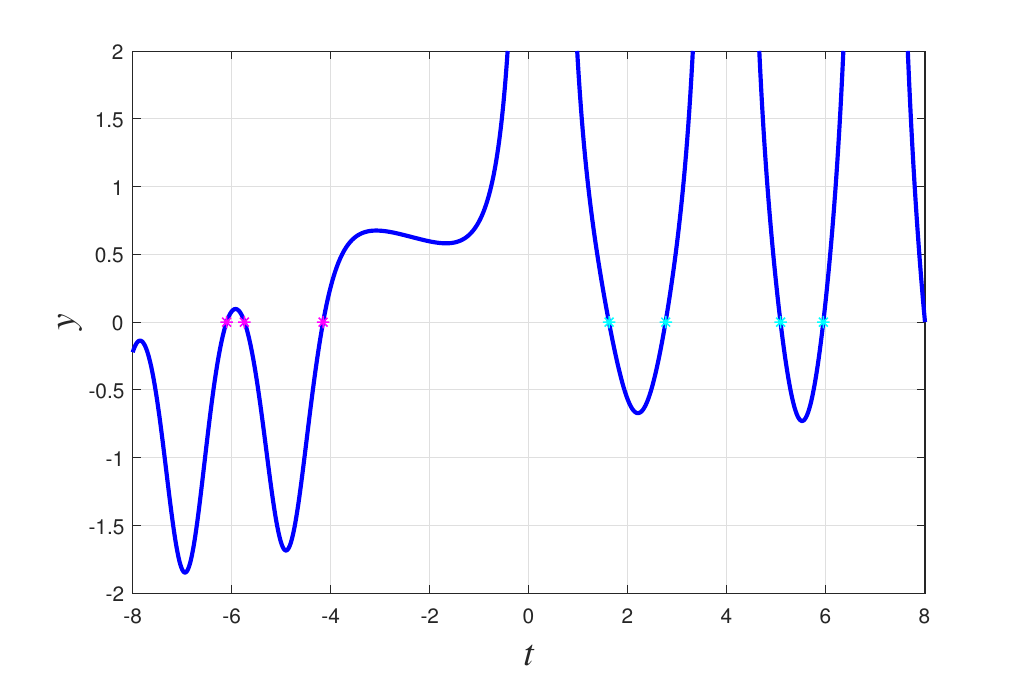}
\caption{
Two types of zeros of oscillatory solutions.
Cyan asterisks ({\color{cyan}$*$}) denote zeros located in the pole region, while magenta asterisks ({\color{magenta}$*$}) indicate zeros located in the oscillatory region.
}
\label{fig-two-types-zeros}
\end{figure}

\begin{remark}
In the above analysis, we have shown the existence of the curves \( \Sigma_{n}^{\pm} \) and provided their asymptotic descriptions as \( n \to \infty \). Theoretically, these descriptions involve two undetermined constants, \( n_{0} \) in Eq.~\eqref{eq-Sigma-n} and \( m_{0} \) in Eq.~\eqref{eq-Sigma-n-extend}. However, numerical simulations suggest that \( n_{0} = m_{0} = 0 \).
\end{remark}

A dramatic difference arises between the two case \( A < -3/2^{2/3} \) and \( A > C_{0} \). As clearly illustrated in Figure~\ref{fig-classification-reigions}, the region corresponding to \( A < -3/2^{2/3} \) (shown in magenta) gives rise to oscillatory solutions, while the region \( A > C_{0} \) (contained within the white blank area) corresponds to singular solutions. These observations are consistent with Theorems~\ref{Thm-stokes-case-I} and~\ref{Thm-stokes-case-III}, and lead to the following corollary.

\begin{corollary}
For any \( A < -3/2^{2/3} \), there exists a constant \( M_{1} > 0 \) such that for all \( \xi > M_{1} \), the PI solutions satisfying \( y(r) = 0 \), \( y'(r) = b \) are of type (A).
Similarly, for any \( A > C_{0} \), there exists a constant \( M_{2} > 0 \) such that for all \( \xi > M_{2} \), the corresponding PI solutions are of type (C).
\end{corollary}

\subsection{Full asymptotic expansions of $\hat{r}_{n}^{\pm}$ and $\hat{b}_{n}^{\pm}$ for specific PI solutions}

In this subsection, we address the connection problem between the large negative asymptotic behavior of PI solutions and the zero parameters \( (r, b) \) for specific solutions. We focus on zeros located in the pole region. Let \( \hat{r}_{n}^{\pm} \) be defined as follows:
\begin{enumerate}
    \item[(i)] For a given type (A) or type (B) solution, \( \hat{r}_{n}^{-} \) and \( \hat{r}_{n}^{+} \) denote the locations of the two zeros situated between the \( n \)-th and \( (n+1) \)-th poles, counted from left to right;
    \item[(ii)] For a given type (C) solution, \( \hat{r}_{n}^{-} \) and \( \hat{r}_{n}^{+} \) denote the \( (2n - 1) \)-th and \( 2n \)-th zeros, respectively, also counted from left to right.
\end{enumerate}
Correspondingly, \( \hat{b}_{n}^{\pm} \) denotes the value of the derivative \( y'(t) \) at \( t = \hat{r}_{n}^{\pm} \).
We aim to derive the full asymptotic expansions of \( \hat{r}_{n}^{\pm} \) and \( \hat{b}_{n}^{\pm} \) as \( n \to \infty \), expressed in terms of the parameters \( d \), \( \theta \), \( h \), \( \rho \), and \( \sigma \). These parameters appear in the asymptotic descriptions of PI solutions as \( t \to -\infty \); see Eqs.~\eqref{eq-behavior-type-A}--\eqref{eq-parameter-rho-sigma}. By observing the connection formulas \eqref{eq-parameter-d-theta}, \eqref{eq-h-s1-s-1}, and \eqref{eq-parameter-rho-sigma}, which relate these parameters to the Stokes multipliers, we conclude that it suffices to construct the full asymptotic expansions of \( \hat{r}_{n} \) and \( \hat{b}_{n} \) in terms of the corresponding Stokes multipliers.

The main idea for establishing the full asymptotic expansions of \( \hat{r}_{n}^{\pm} \) and \( \hat{b}_{n}^{\pm} \) is to resolve the coefficients \( B_{s} \) and the parameter \( \xi \) from Theorem~\ref{Thm-stokes-case-IV}, and substitute the resulting expressions into
\begin{equation}
r = 2C_{0}\xi^{4/5}, \qquad b = \pm 2\xi^{3/5} \left(1 + \sum_{s=1}^{\infty} \frac{B_{s}}{\xi^{s}} \right)^{\frac{1}{2}}.
\end{equation}
Indeed, according to the asymptotics in Eq.~\eqref{eq-stokes-full-asymptotic-case-IV}, we have
\begin{equation}
|s_{1}|\sim \exp\left\{\sum\limits_{s=0}^{\infty}\frac{2\Re\alpha_{s+1}}{\xi^{s}}\right\},\qquad \xi\to+\infty.
\end{equation}
For a specific PI solution, the Stokes multiplier \( s_{1} \) is fixed and explicitly known. Therefore, it follows that
\begin{equation}
4B_{1}\Re\alpha_{1,1}+2\Re\alpha_{1,2}=\log\left(|s_{1}|\right),\quad \Re\alpha_{s}=0.
\end{equation}
Combining it with the expression of $\alpha_{s}$  in Eq.~\eqref{eq-facterization-alphas}, we can explicitly solve for the coefficients \( B_{s} \) as
\begin{equation}\label{eq-solve-Bs-by-s1}
\begin{split}
B_{1}&=\frac{\log\left(|s_{1}|\right)-2\Re\alpha_{1,2}}{4\Re\alpha_{1,1}},\\
B_{s}&=-\frac{\Re\alpha_{s,2}}{\alpha_{1,1}}, \qquad s\geq 2,
\end{split}
\end{equation}
where \( \alpha_{s,1} \) and \( \alpha_{s,2} \) are defined in Eq.~\eqref{eq-def-alphas1-alphas2} with \( A = C_{0} \).
Furthermore, from Theorem~\ref{Thm-stokes-case-IV} and the fact that \( \Im \alpha_{0} < 0 \), we obtain
\begin{equation}
-2(n+n_{1})\pi+\arg{s_{1}}-\frac{\pi}{2}\sim 2\xi\Im\alpha_{0}+2\Im\alpha_{1}+\sum\limits_{s=1}^{\infty}\frac{2\Im\alpha_{s+1}}{\xi^{s}},\quad \xi\to+\infty.
\end{equation}
for some \( n_{1} \in \mathbb{Z} \). Thus, for any fixed \( s_{1} \), there exist two increasing sequences \( \{ \xi_{n}^{+} \} \) and \( \{ \xi_{n}^{-} \} \) such that
\begin{equation}\label{eq-full-asym-expan-xin-by-s1}
\xi:=\xi_{n}^{\pm}\sim \frac{N^{\pm}-\arg{s_1}}{-2\Im\alpha_{0}}\left[1+\sum\limits_{s=2}^{\infty}\frac{\lambda_{s}^{\pm}}{(N^{\pm}-\arg{s_1})^{s}}\right], \quad\text{as } n\to\infty,
\end{equation}
where
\begin{equation}\label{eq-def-N}
\begin{split}
N^{\pm}&=2(n+n_{1})\pi+\frac{\pi}{2}+2\Im\alpha_{1}\\
&=2(n+n_{1})\pi+\frac{\pi}{2}+4B_{1}\Im\alpha_{1,1}(C_{0})+ 2\Im\alpha_{1,2}(C_{0},\pm 1).
\end{split}
\end{equation}
and the coefficients \( \lambda_{s}^{\pm} \) are expressible in terms of \( \Im \alpha_{j} \), for \( j = 0, 1, \ldots, s \).
By substituting Eqs.~\eqref{eq-solve-Bs-by-s1} and~\eqref{eq-full-asym-expan-xin-by-s1} into the expressions for \( r \) and \( b \), we obtain the full asymptotic expansions of \( \hat{r}_{n}^{\pm} \) and \( \hat{b}_{n}^{\pm} \) as \( n \to \infty \). This leads to the following corollary.

\begin{corollary}\label{cor-connection-PI-zero-data}
For any specific PI solution, the asymptotic behavior of the zero parameters \( \hat{r}_{n}^{\pm} \) and \( \hat{b}_{n}^{\pm} \) as \( n \to \infty \) is given by:
\begin{equation}\label{eq-full-expansion-rn-bn}
\begin{split}
\hat{r}_{n}^{\pm}\sim &2C_{0}\left(\frac{N^{\pm}-\arg{s_1}}{-2\Im\alpha_{0}}\right)^{\frac{4}{5}}\left[1+\sum\limits_{s=2}^{\infty}\frac{\gamma_{s}^{\pm}}{(N^{\pm}-\arg{s_1})^{s}}\right] \\
\hat{b}_{n}^{\pm}\sim &\pm2\left(\frac{N^{\pm}-\arg{s_1}}{-2\Im\alpha_{0}}\right)^{\frac{3}{5}}\left[1+\frac{-B_{1}\Im\alpha_{0}}{N^{\pm}-\arg{s_{1}}}+\sum\limits_{s=2}^{\infty}\frac{\beta_{s}^{\pm}}{(N^{\pm}-\arg{s_1})^{s}}\right]
\end{split}
\end{equation}
for some $n_{1}\in\mathbb{Z}$, where $N^{\pm}$ is stated in \eqref{eq-def-N}, $\alpha_{0}$ is defined in \eqref{eq-def-alpha0} with $A=C_{0}$ and $B_{1}$ is given in \eqref{eq-solve-Bs-by-s1}.
Moreover, the coefficients \( \gamma_{s}^{\pm} \) and \( \beta_{s}^{\pm} \) in \eqref{eq-full-expansion-rn-bn} are explicitly determined by the coefficients \( \lambda_{j}^{\pm} \) from Eq.~\eqref{eq-full-asym-expan-xin-by-s1}, for \( j = 2, 3, \dots, s \).

The values of \( |s_{1}| \) and \( \arg s_{1} \) are specified below according to the solution type:
\begin{enumerate}[(i)]
\item  For type (A) solutions, {\it i.e.} when $\Im s_{0}=1+s_{2}s_{3}>0$, we have
\begin{equation}\label{eq-|s1|-by-s2-A}
|s_{1}|=\frac{\sqrt{1-2|s_{2}|\cos(\frac{\pi}{2}-\arg{s_{2}})+|s_{2}|^2}}{|s_{0}|}
\end{equation}
and
\begin{equation}\label{eq-args1-by-s2-A}
\arg{s_{1}}-\frac{\pi}{2}=\arctan\left(\frac{-|s_{2}|\sin(\frac{\pi}{2}-\arg{s_{2}})}{1-|s_{2}|\cos(\frac{\pi}{2}-\arg{s_{2}})}\right);
\end{equation}

\item  For type (B) solutions, {\it i.e.} when $\Im s_{0}=(1+s_{2}s_{3})=0$, we have
\begin{equation}\label{eq-|s1|-by-h}
|s_{1}|=\frac{\sqrt{h^2+1}}{2}
\end{equation}
and
\begin{equation}\label{eq-args1-by-h}
\arg{s_{1}}-\frac{\pi}{2}=\arctan(-h);
\end{equation}

\item  For type (C) solutions, {\it i.e.} when $\Im s_{0}=(1+s_{2}s_{3})<0$, we have
\begin{equation}\label{eq-|s1|-by-s2-C}
|s_{1}|=\frac{\sqrt{1-2|s_{2}|\cos(\frac{\pi}{2}-\arg{s_{2}})+|s_{2}|^2}}{|s_{0}|}
\end{equation}
and
\begin{equation}\label{eq-args1-by-s2-C}
\arg{s_{1}}-\frac{\pi}{2}=\left\{\begin{aligned}
&\arctan\left(\frac{|s_{2}|\sin(\frac{\pi}{2}-\arg{s_{2}})}{|s_{2}|\cos(\frac{\pi}{2}-\arg{s_{2}})-1}\right), && |s_{2}|\cos(\frac{\pi}{2}-\arg{s_{2}})>1,\\
&\pi+\arctan\left(\frac{|s_{2}|\sin(\frac{\pi}{2}-\arg{s_{2}})}{|s_{2}|\cos(\frac{\pi}{2}-\arg{s_{2}})-1}\right), && |s_{2}|\cos(\frac{\pi}{2}-\arg{s_{2}})<1.
\end{aligned}\right.
\end{equation}
\end{enumerate}
\end{corollary}

\begin{proof}
To prove the corollary, it suffices to verify the formulas \eqref{eq-|s1|-by-s2-A}–\eqref{eq-args1-by-s2-C} for \( |s_{1}| \) and \( \arg s_{1} \) in each of the three solution types.

\smallskip
\noindent
\textbf{Type (A).} From the identities \( s_{2} = -\overline{s_{3}} \) and \( s_{3} = s_{-2} = i(1 + s_{0}s_{1}) \) and the condition \( \Im s_{0} > 0 \), we have
\begin{equation}
s_{1} = \frac{s_{3} - i}{i s_{0}} = i \frac{|s_{3}| e^{i(\arg s_{3} - \frac{\pi}{2})} - 1}{-|s_{0}|}
= i \frac{1 - |s_{2}| e^{i(\frac{\pi}{2} - \arg s_{2})}}{|s_{0}|}.
\end{equation}
This leads directly to Eqs.~\eqref{eq-|s1|-by-s2-A} and \eqref{eq-args1-by-s2-A}, using the fact that \( 1 - |s_{2}| \cos(\tfrac{\pi}{2} - \arg s_{2}) > 0 \) in this case.

\smallskip
\noindent
\textbf{Type (B).} From \eqref{eq-classifed-by-stokes} and \eqref{eq-sk-s-k-relation}, we know that\( s_{0} = 0 \). Then it follows from the constraint relation \eqref{eq-constraints-stokes-multipliers} that \( s_{2} = s_{3} = i \), and further that \( s_{1} + s_{-1} = i \). Using the fact that \( s_{1} - s_{-1} = h \), we solve:
\[
s_{1} = \frac{h + i}{2} = i \frac{1 - hi}{2}.
\]
This immediately yields Eqs.~\eqref{eq-|s1|-by-h} and \eqref{eq-args1-by-h}.

\smallskip
\noindent
\textbf{Type (C).} The analysis is similar to the type (A) case, except that \( \Im s_{0} < 0 \), and the quantity \( 1 - |s_{2}| \cos(\tfrac{\pi}{2} - \arg s_{2}) \) may be negative. As such, when determining \( \arg s_{1} \), we must carefully choose the correct branch of the argument to ensure continuity. This leads to the two cases in Eq.~\eqref{eq-args1-by-s2-C}.
\end{proof}

Corollary \ref{cor-connection-PI-zero-data} resolves the connection problem for the PI equation by relating the large negative asymptotics to the local asymptotic behavior \eqref{eq-behavior-at-zero} near large positive zeros. For a given solution of the PI equation, specifying either the parameters $(d, \theta)$ in \eqref{eq-parameter-d-theta}, the constant $h$ in \eqref{eq-h-s1-s-1}, or the pair $(\rho, \sigma)$ in \eqref{eq-parameter-rho-sigma}, explicitly determines the corresponding Stokes multipliers. Then, by Corollary \ref{cor-connection-PI-zero-data}, one obtains the asymptotic behavior of the zero data associated with the large positive zeros of this solution.

\begin{table}[htp]
  \centering
  \caption{Comparison between the asymptotic and numerical values of $\hat{r}_{n}^{+}$ and $\hat{b}_{n}^{+}$ for the real tritronqu\'{e}e PI solution with $s_{0}=s_{1}=s_{-1}=i, s_{2}=s_{-2}=0$}\label{table-comparison-rn-bn-type-A}
  \begin{tabular}{c|cc|cc|cc}
  \toprule
\#	&   \multicolumn{2}{c|}{asymptotic values}		&\multicolumn{2}{c|}{numerical values} 	& \multicolumn{2}{c}{relative  errors}   \\
\midrule
$n$	& $\hat{r}_{n}^{+}\quad~$ & $\hat{b}_{n}^{+}$		& $\hat{r}_{n}^{+}\quad~$	& $\hat{b}_{n}^{+}$		& $\Delta \hat{r}_{n}^{+}/\hat{r}_{n}^{+}$  		& $\Delta \hat{b}_{n}^{+}/\hat{b}_{n}^{+}$	  \\
\hline
 1	& $ 4.512112$ & $2.092355$ & $ 4.482589 $ & $2.095210$ & $0.006543$	& $-0.001365$\\
 2	& $ 7.488588$ & $3.123150$ & $ 7.473984$ & $3.124409$ & $0.001950$	& $-0.000403$\\
 3	& $ 10.187133$ & $3.962947$ & $ 10.177384$ & $3.963772$ & $0.000957$	& $-0.000208$\\
 4	& $12.715558$ & $4.697489$ & $12.708160$ & $4.698113$ & $0.000582$	& $-0.000133$\\
 5	& $15.123217$ & $5.362188$	& $15.117192$ & $5.362694$ & $0.000398$	& $-0.000094$\\
\bottomrule
  \end{tabular}
\end{table}

\begin{table}[htp]
  \centering
  \caption{Comparison between the asymptotic and numerical values of $\hat{r}_{n}^{+}$ and $\hat{b}_{n}^{+}$ for the real PI solution with $s_{k}=-2i\cos{\frac{\pi}{5}}, k=0,1,2,3,4$.}\label{table-comparison-rn-bn-type-C}
  \begin{tabular}{c|cc|cc|cc}
  \toprule
\#	&   \multicolumn{2}{c|}{asymptotic values}		&\multicolumn{2}{c|}{numerical values} 	& \multicolumn{2}{c}{relative  errors}   \\
\midrule
$n$	& $\hat{r}_{n}^{+}$ & $\hat{b}_{n}^{+}$		& $\hat{r}_{n}^{+}$	& $\hat{b}_{n}^{+}$		& $\Delta \hat{r}_{n}^{+}/\hat{r}_{n}^{+}$  		& $\Delta \hat{b}_{n}^{+}/\hat{b}_{n}^{+}$	\\ \hline
 1	& $ 2.599108$ & $1.507663$ & $ 2.575996$ & $1.527218$ & $0.008892$	& $-0.012971$\\
 2	& $ 5.857157$ & $2.699281$ & $ 5.850983$ & $2.704178$ & $0.001054$	& $-0.001814$\\
 3	& $ 8.691611$ & $3.607257$ & $ 8.688664$ & $3.609591$ & $0.000339$	& $-0.000647$\\
 4	& $11.307599$ & $4.382464$ & $11.305909$ & $4.383839$ & $0.000149$	& $-0.000313$\\
 5	& $13.778912$ & $5.075135$	& $13.777867$ & $5.076029$ & $0.000076$	& $-0.000176$\\
\bottomrule
  \end{tabular}
\end{table}

\begin{remark}
Although there is an undetermined constant $n_{1}$ in \eqref{eq-full-expansion-rn-bn}, numerical simulations confirm that $n_{1} = 0$. Table \ref{table-comparison-rn-bn-type-A} presents a comparison between the numerical and asymptotic values of $\hat{r}_{n}^{+}$ and $\hat{b}_{n}^{+}$ for the real tritronqu\'ee solution, which can be regarded as a special case of type (A). The numerical values of $\hat{r}_{n}^{+}$ and $\hat{b}_{n}^{+}$ are obtained using an improved Runge–Kutta method implemented in MATLAB, while their asymptotic values are taken from \eqref{eq-full-expansion-rn-bn}, omitting the high-order terms with $s \geq 2$.
In Table \ref{table-comparison-rn-bn-type-C}, we compare the numerical and asymptotic values of $\hat{r}_{n}^{+}$ and $\hat{b}_{n}^{+}$ for a specific type (C) solution with $(p, H) = (0, 0)$, where $p$ is a pole of the solution and $H$ is the parameter in the following Laurent expansion near $t = p$:
\begin{equation}\label{eq-Laurent-series}
y(t) = \frac{1}{(t - p)^2} - \frac{p}{10}(t - p)^2 - \frac{1}{6}(t - p)^3 + H(t - p)^4 + \cdots.
\end{equation}
The Stokes multipliers of this solution are given by $s_{k} = -2i \cos{\frac{\pi}{5}}$ for $k = 0, 1, 2, 3, 4$; see \cite{Long-Xia}.
\end{remark}

\section{Asymptotic matching and proof of Theorems \ref{Thm-stokes-case-I}-\ref{Thm-stokes-case-IV}}
\label{sec:proof-theorem}

In this section, we rigorously derive the full asymptotic expansions of the Stokes multipliers $\{s_k\}$ as $\xi \to +\infty$, thereby completing the proofs of Theorems~\ref{Thm-stokes-case-I}--\ref{Thm-stokes-case-IV}. Our approach is based on asymptotic matching between the uniform approximation of the solution $W(\zeta,\xi)$ given in Lemma~\ref{lem-higher-order-approximation-ODE-1} and the canonical solutions $\{\widehat{\Phi}_k\}$ characterized by the expansions in~\eqref{eq-canonical-solutions-hat-Phi}.

Two key analytical ingredients underpin our strategy. First, the Airy-type asymptotics for $W(\zeta,\xi)$ provided by Lemma~\ref{lem-higher-order-approximation-ODE-1} are uniformly valid across adjacent Stokes sectors. Second, the associated error terms decay sufficiently fast, allowing for the extraction of precise higher-order corrections. These features, combined with the classical Stokes phenomenon for the Airy function, facilitate the systematic derivation of the full asymptotic structure of the Stokes multipliers.

Although our methodology builds upon the framework developed in earlier works~\cite{Long-Jiang-Li, LongLi, LongLiWang}, the current analysis introduces several technical refinements.
For brevity and to avoid duplication, we omit standard derivations and focus on the novel aspects of the computation. Readers interested in foundational techniques and auxiliary arguments are referred to~\cite{Long-Jiang-Li, LongLi, LongLiWang}.

\

\textbf{Proof of Theorem \ref{Thm-stokes-case-I}}

According to Figure~\ref{Figure-stokes-geometry}, when $A < -3/2^{2/3}$, we can apply Lemma~\ref{lem-higher-order-approximation-ODE-1} to obtain the uniform asymptotic behavior of $W$ in a neighborhood of the turning point $z = z_1$ and along the adjacent Stokes curves $\gamma_0$ and $\gamma_1$ that emanate from it. This facilitates the asymptotic computation of the Stokes multiplier $s_0$.

We begin by analyzing the asymptotic expansions of $\Ai_n(\xi^{2/3}\zeta)$ and $\Bi_n(\xi^{2/3}\zeta)$, and matching these with the canonical solutions $\widehat{\Phi}_k$ of the Lax pair. Recall the well-known asymptotics of the Airy functions for large complex arguments (see~\cite[Eqs.~(9.2.12), (9.7.5), (9.2.10)]{NIST-handbook}):
\begin{eqnarray}\label{eq-asym-Ai}
\begin{cases}
\Ai(z)\sim
\frac{1}{2\sqrt{\pi}}z^{-\frac{1}{4}}e^{-\frac{2}{3}z^{\frac{3}{2}}},
\quad & \arg{z}\in(-\pi,\pi),\\
\Ai(z)\sim
\frac{1}{2\sqrt{\pi}}z^{-\frac{1}{4}}e^{-\frac{2}{3}z^{\frac{3}{2}}}
+\frac{i}{2\sqrt{\pi}}z^{-\frac{1}{4}}e^{\frac{2}{3}z^{\frac{3}{2}}},
\quad & \arg{z}\in\left(\frac{\pi}{3},\frac{5\pi}{3}\right),\\
\Ai(z)\sim
\frac{1}{2\sqrt{\pi}}z^{-\frac{1}{4}}e^{-\frac{2}{3}z^{\frac{3}{2}}}
-\frac{i}{2\sqrt{\pi}}z^{-\frac{1}{4}}e^{\frac{2}{3}z^{\frac{3}{2}}},
\quad &\arg{z}\in\left(-\frac{5\pi}{3},-\frac{\pi}{3}\right)
\end{cases}
\end{eqnarray}
and
\begin{eqnarray}\label{eq-asym-Bi}
\begin{cases}
\Bi(z)\sim
\frac{i}{2\sqrt{\pi}}z^{-\frac{1}{4}}e^{-\frac{2}{3}z^{\frac{3}{2}}}
+\frac{1}{\sqrt{\pi}}z^{-\frac{1}{4}}e^{\frac{2}{3}z^{\frac{3}{2}}},
\quad&\arg{z}\in\left(-\frac{\pi}{3},\pi\right),\\
\Bi(z)\sim
\frac{i}{2\sqrt{\pi}}z^{-\frac{1}{4}}e^{-\frac{2}{3}z^{\frac{3}{2}}}
+\frac{1}{2\sqrt{\pi}}z^{-\frac{1}{4}}e^{\frac{2}{3}z^{\frac{3}{2}}},
\quad & \arg{z}\in\left(\frac{\pi}{3},\frac{5\pi}{3}\right),\\
\Bi(z)\sim
-\frac{i}{2\sqrt{\pi}}z^{-\frac{1}{4}}e^{-\frac{2}{3}z^{\frac{3}{2}}}
+\frac{1}{\sqrt{\pi}}z^{-\frac{1}{4}}e^{\frac{2}{3}z^{\frac{3}{2}}},
\quad & \arg{z}\in\left(-\pi,\frac{\pi}{3}\right).
\end{cases}
\end{eqnarray}

For $\lambda \to \infty$ with $\arg \lambda \sim \pi/5$, we use the relation $\lambda = \xi^{2/5} z$ and note from~\eqref{eq-asymp-relation-zeta-z} that $\arg \zeta \sim \pi/3$. Substituting \eqref{eq-asymp-relation-zeta-z} into~\eqref{eq-asym-Ai} and~\eqref{eq-asym-Bi}, and using the relation $\lambda = \xi^{2/5} z$ again, we obtain (for $n \in \mathbb{N}^{+}$):
\begin{eqnarray}\label{eq-Ai-Bi-pi/5}
\begin{cases}
\left(\frac{\zeta}{f(z)}\right)^{\frac{1}{4}}\Ai_{n}(\xi^{\frac{2}{3}}\zeta)=\left(\frac{\zeta}{f(z)}\right)^{\frac{1}{4}}\left(\frac{d\widehat{\zeta}}{d\zeta}\right)^{-\frac{1}{2}}\Ai\left(\xi^{\frac{2}{3}}\widehat{\zeta}\right)\sim c_{1}\frac{-1}{\sqrt{2}}\lambda^{-\frac{3}{4}}e^{-\frac{4}{5}\lambda^{\frac{5}{2}}},\\
\left(\frac{\zeta}{f(z)}\right)^{\frac{1}{4}}\Bi_{n}(\xi^{\frac{2}{3}}\zeta)=\left(\frac{\zeta}{f(z)}\right)^{\frac{1}{4}}\left(\frac{d\widehat{\zeta}}{d\zeta}\right)^{-\frac{1}{2}}\Bi\left(\xi^{\frac{2}{3}}\widehat{\zeta}\right)\sim ic_{1}\frac{-1}{\sqrt{2}}\lambda^{-\frac{3}{4}}e^{-\frac{4}{5}\lambda^{\frac{5}{2}}}+2c_{2}\frac{1}{\sqrt{2}}\lambda^{-\frac{3}{4}}e^{\frac{4}{5}\lambda^{\frac{5}{2}}}
\end{cases}
\end{eqnarray}
as $\lambda\to\infty$ (and $\zeta\to\infty$ accordingly), where the constants $c_1$, $c_2$ are defined by:
\begin{equation}\label{eq-def-c1-c2}
\begin{split}
c_{1}&=\frac{-1}{2\sqrt{\pi}}\xi^{-\frac{2}{15}}\exp\left\{-\sum\limits_{s=0}^{n}\alpha_{s}\xi^{-s+1}\right\},\\
c_{2}&=\frac{1}{2\sqrt{\pi}}\xi^{-\frac{2}{15}}\exp\left\{\sum\limits_{s=0}^{n}\alpha_{s}\xi^{-s+1}\right\}
\end{split}
\end{equation}
and $\alpha_{s}$ is defined in \eqref{eq-facterization-alphas}.

Similarly when $\lambda\to\infty$ with $\arg\lambda\sim-\frac{\pi}{5}$, we have
\begin{eqnarray}\label{eq-Ai-Bi--pi/5}
\begin{cases}
\left(\frac{\zeta}{f(z)}\right)^{\frac{1}{4}}\Ai_{n}(\xi^{\frac{2}{3}}\zeta)=\left(\frac{\zeta}{f(z)}\right)^{\frac{1}{4}}\left(\frac{d\widehat{\zeta}}{d\zeta}\right)^{-\frac{1}{2}}\Ai\left(\xi^{\frac{2}{3}}\widehat{\zeta}\right)\sim c_{1}\frac{-1}{\sqrt{2}}\lambda^{-\frac{3}{4}}e^{-\frac{4}{5}\lambda^{\frac{5}{2}}},\\
\left(\frac{\zeta}{f(z)}\right)^{\frac{1}{4}}\Bi_{n}(\xi^{\frac{2}{3}}\zeta)=\left(\frac{\zeta}{f(z)}\right)^{\frac{1}{4}}\left(\frac{d\widehat{\zeta}}{d\zeta}\right)^{-\frac{1}{2}}\Bi\left(\xi^{\frac{2}{3}}\widehat{\zeta}\right)\sim -ic_{1}\frac{-1}{\sqrt{2}}\lambda^{-\frac{3}{4}}e^{-\frac{4}{5}\lambda^{\frac{5}{2}}}+2c_{2}\frac{1}{\sqrt{2}}\lambda^{-\frac{3}{4}}e^{\frac{4}{5}\lambda^{\frac{5}{2}}}
\end{cases}
\end{eqnarray}

Now, match these expansions with the canonical solution $\widehat{\Phi}_k$. From~\eqref{eq-canonical-solutions}, a direct calculation shows
\begin{equation}\label{eq-asym-Phi21-Phi22}
\left((\widehat{\Phi}_{k})_{21},(\widehat{\Phi}_{k})_{22}\right)\sim \left(\frac{1}{\sqrt{2}}\lambda^{-\frac{3}{4}}e^{\frac{4}{5}\lambda^{\frac{5}{2}}},\,
\frac{-1}{\sqrt{2}}\lambda^{-\frac{3}{4}}e^{-\frac{4}{5}\lambda^{\frac{5}{2}}}\right),
\qquad k\in\mathbb{Z}
\end{equation}
as $\lambda\to\infty$.
Substituting \eqref{eq-Ai-Bi-pi/5} into \eqref{eq-higher-approximation-ODE-1}, noting that $Y=\left(\frac{\zeta}{f(z)}\right)^{\frac{1}{4}}$ and $Y$ is the first component of $\widehat{\Phi}$,
and then comparing the corresponding result with \eqref{eq-asym-Phi21-Phi22},
we get
\begin{equation}\label{eq-Y-pi/5}
\begin{split}
W= &[C_{1}+r_{1,1}(\xi)]c_{1}(\widehat{\Phi}_{1})_{22}+[C_{2}+r_{2,1}(\xi)][ic_{1}(\widehat{\Phi}_{1})_{22}+2c_{2}(\widehat{\Phi}_{1})_{21}]\\
=&[C_{2}+r_{2,1}(\xi)]2c_{2}(\widehat{\Phi}_{1})_{21}+\{[C_{1}+r_{1,1}(\xi)]c_{1}+[C_{2}+r_{2,1}(\xi)]ic_{1}\}(\widehat{\Phi}_{1})_{22}
\end{split}
\end{equation}
with $\lambda\in\Omega_{1}$. Since the Stokes multipliers $s_{k}$'s depend on $\xi$,
hence we will only take the limit $\lambda\to\infty$
when we do the asymptotic matching between the Airy functions and $\Phi_{k}$'s. In a similar manner, we get
\begin{equation}\label{eq-Y--pi/5}
\begin{split}
W=[C_{2}+r_{2,0}(\xi)]2c_{2}(\widehat{\Phi}_{0})_{21}+\{[C_{1}+r_{1,0}(\xi)]c_{1}-[C_{2}+r_{2,0}(\xi)]ic_{1}\}(\widehat{\Phi}_{0})_{22}
\end{split}
\end{equation}
with $\lambda\in\Omega_{0}$.
Combining \eqref{eq-Y-pi/5} with \eqref{eq-Y--pi/5}, and noting that
$$\left((\widehat{\Phi}_{1})_{21},(\widehat{\Phi}_{1})_{22}\right)=\left((\widehat{\Phi}_{0})_{21},(\widehat{\Phi}_{0})_{22}\right)\left(\begin{matrix}1&0\\s_{0}&1\end{matrix}\right),$$
we conclude that $r_{2,0}(\xi)=r_{2,1}(\xi)$ and
\begin{equation}
s_{0}=-\frac{ic_{1}}{c_{2}}\left(1+\frac{r_{1,0}(\xi)-r_{1,1}(\xi)}{2(C_{2}+r_{2,0}(\xi))}\right).
\end{equation}
Finally, substituting~\eqref{eq-def-c1-c2} into this identity and recalling that $r_{1,k} = \mathcal{O}(\xi^{-2n})$ as $\xi \to +\infty$ for $k=1,2$ and arbitrary $n \in \mathbb{N}^{+}$, we obtain the desired asymptotic expression~\eqref{eq-s0-case-I}. \qed

\

\textbf{Proof of Theorem \ref{Thm-stokes-case-II}}

Referring to Figure~\ref{Figure-stokes-geometry}, when $A\in(-3/2^{2/3}, C_{0})$, we apply Lemma~\ref{lem-higher-order-approximation-ODE-1} to obtain the uniform asymptotic behavior of $W$ in a neighborhood of the turning point $z = z_1$ and along the adjacent Stokes curves $\gamma_1$ and $\gamma_2$. In contrast to Theorem~\ref{Thm-stokes-case-I}, we now compute the Stokes multiplier $s_1$ rather than $s_0$.

As $\lambda \to \infty$ with $\arg \lambda \sim \frac{3\pi}{5}$, we observe that $\arg \zeta \sim \pi$. Proceeding analogously to the derivation of~\eqref{eq-Ai-Bi-pi/5}, and using the asymptotic expansions of the Airy functions near $\arg z \sim \pi$ from~\eqref{eq-asym-Ai} and~\eqref{eq-asym-Bi}, we obtain
\begin{eqnarray}\label{eq-Ai-Bi-3pi/5}
\begin{cases}
\left(\frac{\zeta}{f(z)}\right)^{\frac{1}{4}}\Ai_{n}(\xi^{\frac{2}{3}}\zeta)=\left(\frac{\zeta}{f(z)}\right)^{\frac{1}{4}}\left(\frac{d\widehat{\zeta}}{d\zeta}\right)^{-\frac{1}{2}}\Ai\left(\xi^{\frac{2}{3}}\widehat{\zeta}\right)\sim c_{1}\frac{-1}{\sqrt{2}}\lambda^{-\frac{1}{4}}e^{-\frac{4}{5}\lambda^{\frac{5}{2}}}+ic_{2}\frac{1}{\sqrt{2}}\lambda^{-\frac{1}{4}}e^{\frac{4}{5}\lambda^{\frac{5}{2}}},\\
\left(\frac{\zeta}{f(z)}\right)^{\frac{1}{4}}\Bi_{n}(\xi^{\frac{2}{3}}\zeta)=\left(\frac{\zeta}{f(z)}\right)^{\frac{1}{4}}\left(\frac{d\widehat{\zeta}}{d\zeta}\right)^{-\frac{1}{2}}\Bi\left(\xi^{\frac{2}{3}}\widehat{\zeta}\right)\sim ic_{1}\frac{-1}{\sqrt{2}}\lambda^{-\frac{1}{4}}e^{-\frac{4}{5}\lambda^{\frac{5}{2}}}+c_{2}\frac{1}{\sqrt{2}}\lambda^{-\frac{1}{4}}e^{\frac{4}{5}\lambda^{\frac{5}{2}}}
\end{cases}
\end{eqnarray}
as $\lambda\to\infty$ (and $\zeta\to\infty$ accordingly). Substituting \eqref{eq-Ai-Bi-3pi/5} into \eqref{eq-higher-approximation-ODE-1}, noting that $Y=\left(\frac{\zeta}{f(z)}\right)^{\frac{1}{4}}$ and $Y$ is the first component of $\widehat{\Phi}$,
and then comparing the corresponding result with \eqref{eq-asym-Phi21-Phi22},
we get
\begin{equation}\label{eq-Y-3pi/5-C}
\begin{split}
W=&\{[C_{1}+r_{1,2}(\xi)]ic_{2}+[C_{2}+r_{2,2}(\xi)]c_{2}\}(\widehat{\Phi}_{2})_{21}\\
&+\{[C_{1}+r_{1,2}(\xi)]c_{1}+[C_{2}+r_{2,2}(\xi)]ic_{1}\}(\widehat{\Phi}_{2})_{22}
\end{split}
\end{equation}
$\lambda\in\Omega_{2}$.
Combining \eqref{eq-Y-pi/5} with \eqref{eq-Y-3pi/5-C}, and noting that
$$\left((\widehat{\Phi}_{2})_{21},(\widehat{\Phi}_{2})_{22}\right)=\left((\widehat{\Phi}_{1})_{21},(\widehat{\Phi}_{1})_{22}\right)\left(\begin{matrix}1&s_{1}\\0&1\end{matrix}\right),$$
we further obtain
\begin{equation}
r_{1,1}(\xi)-r_{1,2}(\xi)=i(r_{2,2}(\xi)-r_{2,1}(\xi))
\end{equation}
and
\begin{equation}
\begin{split}
s_{1}&=\frac{(C_{2}-iC_{1})c_{2}+(2r_{2,1}(\xi)-ir_{1,2}(\xi)-r_{2,2}(\xi))c_{2}}{(C_{1}+iC_{2})c_{1}+r_{1,2}(\xi)c_{1}+r_{2,2}(\xi)ic_{1}}\\
&=-\frac{ic_{2}}{c_{1}}+\frac{2(r_{2,1}(\xi)-r_{2,2}(\xi))c_{2}}{(C_{1}+iC_{2}+r_{1,2}(\xi)+ir_{2,2}(\xi))c_{1}}.
\end{split}
\end{equation}
Since the Stokes multipliers $s_{k}$'s depend on $\xi$,
hence we will only take the limit $\lambda\to\infty$
when we do the asymptotic matching between the Airy functions and $\Phi_{k}$'s.

By Lemma~\ref{lem-higher-order-approximation-ODE-1}, $r_{j,k}(\xi)=\mathcal{O}\left(\frac{|C_{1}|+|C_{2}|}{\xi^{n}}\right)$ as $\xi\to+\infty$ for $j=1,2$ and $k=1,2$. Hence,
\begin{equation}
s_{1}=i \exp\left\{\sum\limits_{s=0}^{2n}2\alpha_{s}(A_{0},B(\xi))\xi^{-s+1}\right\}\left(1+\mathcal{O}(\xi^{-2n})\right)
\end{equation}
as $\xi\to+\infty$.

The asymptotic expansion of $s_{-1}$ in \eqref{eq-stokes-full-asymptotic-case-II}
follows from the conjugate relation $s_{-1}=\overline{s_{1}}$, and the asymptotic expansions for other stokes multipliers follow from the constraint $s_{k}=i(1+s_{k+2}s_{k+3})$, along with the fact that $\Re\alpha_{0}>0$. \qed

\begin{remark}
Theorems~\ref{Thm-stokes-case-III} and~\ref{Thm-stokes-case-IV} can be proved by following the same strategy as in the proof of Theorem~\ref{Thm-stokes-case-II}. Specifically, the procedure for deriving the asymptotic expansion of the Stokes multiplier $s_{1}$ remains unchanged. However, when computing the asymptotic expansions of the remaining Stokes multipliers, it is important to take into account the behavior of the leading exponent $\alpha_{0}$. Namely, we have $\Re \alpha_{0} = 0$ when $A = C_{0}$, and $\Re \alpha_{0} < 0$ when $A> C_{0}$. These distinctions significantly affect the asymptotic contributions of the associated terms.
\end{remark}

\appendix

\section{Monodromy theory of PI}
\label{sec:AppA}

The Painlev\'e equations can be formulated as compatibility conditions of associated Lax pairs. In particular, the first Painlev\'e equation (PI) admits the following Lax pair representation (see \cite{Kapaev-Kitaev-1993}):
\begin{equation}\label{lax pair-I}
\left\{\begin{aligned}
\frac{\partial\Psi}{\partial\lambda}
&=\left\{(4\lambda^4+t+2y^2)\sigma_{3}-i(4y\lambda^2+t+2y^2)\sigma_{2}
-\left(2y_{t}\lambda+\frac{1}{2\lambda}\right)\sigma_{1}\right\}\Psi
:=\mathcal{A}(\lambda)\Psi,
\\
\frac{\partial\Psi}{\partial t}
&=\left\{\left(\lambda+\frac{y}{\lambda}\right)\sigma_{3}-\frac{iy}{\lambda}\sigma_{2}\right\}\Psi
:=\mathcal{B}(\lambda)\Psi,
\end{aligned}\right.
\end{equation}
where $y_{t}=\frac{dy}{dt}$ and
\[
\sigma_{1}=\begin{bmatrix}0&1\\1&0\end{bmatrix},\quad \sigma_{2}=\begin{bmatrix}0&-i\\i&0\end{bmatrix},\quad
\sigma_{3}=\begin{bmatrix}1&0\\0&-1\end{bmatrix}
\]
stands for the Pauli matrices.
The compatibility condition for this Lax pair,
\[
\frac{\partial^2 \Psi}{\partial t \, \partial \lambda} = \frac{\partial^2 \Psi}{\partial \lambda \, \partial t},
\]
leads to that \( y = y(t) \) statisfies the PI equation~\eqref{PI equation}.

Under the transformation
\begin{equation}\label{eq-transform-canonical solution}
\Phi(\lambda)=\lambda^{\frac{1}{4}\sigma_{3}}\frac{\sigma_{3}+\sigma_{1}}{\sqrt{2}}\Psi(\sqrt{\lambda}),
\end{equation}
the first equation in (\ref{lax pair-I}) becomes
\begin{equation}\label{eq-fold-Lax-pair}
\frac{\partial\Phi}{\partial\lambda}=\left(\begin{matrix}y_{t}&2\lambda^{2}+2y\lambda-t+2y^2\\2(\lambda-y)&-y_{t}\end{matrix}\right)\Phi.
\end{equation}
Following~\cite{Kapaev-Kitaev-1993} (see also~\cite{AAKapaev-2004}), the only singularity of Eq.~\eqref{eq-fold-Lax-pair} is the irregular singular point at \( \lambda = \infty \). There exist canonical solutions \( \Phi_k(\lambda) \), indexed by \( k \in \mathbb{Z} \), with the following asymptotic behavior:
\begin{equation}\label{eq-canonical-solutions}
\Phi_{k}(\lambda,t)
=\lambda^{\frac{1}{4}\sigma_{3}}\frac{\sigma_{3}+\sigma_{1}}{\sqrt{2}}
\left(I+\frac{\mathcal{H}}{\lambda}+\mathcal{O}\left(\frac{1}{\lambda^2}\right)\right)e^{(\frac{4}{5}\lambda^{\frac{5}{2}}+t\lambda^{\frac{1}{2}})\sigma_{3}},
\qquad\lambda\rightarrow\infty,\quad \lambda\in\Omega_{k},
\end{equation}
uniformly for all $t$ bounded away from the poles of the PI solution.  Here, the coefficient matrix is \[
\mathcal{H}=-(\frac{1}{2}y_{t}^2-2y^3-ty)\sigma_{3},
\]
and the canonical sectors are defined by
\[
\Omega_{k}=\left\{\lambda\in\mathbb{C}:~\arg \lambda\in \left(-\frac{3\pi}{5}+\frac{2k\pi}{5},\frac{\pi}{5}+\frac{2k\pi}{5}\right)\right\}, \qquad k\in\mathbb{Z}.
\]
These canonical solutions are related by Stokes matrices:
\begin{equation}\label{eq-Stokes-matrices}
\Phi_{k+1}=\Phi_{k}S_{k},\quad S_{2k-1}=\left(\begin{matrix}1&s_{2k-1}\\0&1\end{matrix}\right),\quad S_{2k}=\left(\begin{matrix}1&0\\s_{2k}&1\end{matrix}\right),
\end{equation}
where \( s_k \in \mathbb{C} \) are the \emph{Stokes multipliers}.
By the isomonodromy condition, the \( s_k \)'s are independent of both \( \lambda \) and \( t \).
The Stokes multipliers satisfy the following constraints:
\begin{equation}\label{eq-constraints-stokes-multipliers}
s_{k+5}=s_{k}\quad \text{and}\quad s_{k}=i(1+s_{k+2}s_{k+3}),
\qquad k\in\mathbb{Z}.
\end{equation}
Furthermore, regarding $s_{k}$'s as functions of $(t,y(t),y'(t))$, they satisfy the symmetry relation (see~\cite[Eq.~(13)]{AAKapaev-1988}):
\begin{equation}\label{eq-sk-s-k-relation}
s_{k}\left (t,y(t),y'(t)\right)=-\overline{s_{-k}\left (\bar{t},\overline{y(t)},\overline{y'(t)}\right )}, \qquad k\in\mathbb{Z},
\end{equation}
where $\bar{z}$ denotes the complex conjugate of $z$.
From Eq.~\eqref{eq-constraints-stokes-multipliers}, it follows that the full set of Stokes multipliers is determined by any two consecutive ones. Detailed derivations of the asymptotic expansion~\eqref{eq-canonical-solutions}, the Stokes matrices~\eqref{eq-Stokes-matrices}, and the constraint~\eqref{eq-constraints-stokes-multipliers}, along with further discussion of the associated Lax pairs, can be found in~\cite{FAS-2006}.

For further simplification, we introduce the transformation
\begin{equation}\label{eq-def-hat{Phi}}
\widehat{\Phi}(\lambda,t)=G(\lambda,t)\Phi(\lambda,t)
\end{equation}
where the matrix \( G(\lambda, t) \) is defined by
\begin{equation}\label{def-G(lambda,t)}
G(\lambda,t)
=\begin{bmatrix}0&1\\ 1&-\frac{1}{2}\left(-y_{t}+\frac{1}{2(\lambda-y)}\right)\end{bmatrix}
(\lambda-y)^{\frac{\sigma_{3}}{2}}.
\end{equation}
Under this transformation, the function \( \widehat{\Phi}(\lambda, t) \) satisfies the system
\begin{equation}\label{eq-system-hat-Phi}
\frac{d}{d\lambda}\widehat{\Phi}(\lambda,t)
=\begin{bmatrix}0&2\\V(\lambda,t)&0\end{bmatrix}\widehat{\Phi}(\lambda,t),
\end{equation}
where
\[
2V(\lambda,t)=y_{t}^{2}+4\lambda^{3}+2\lambda t-2y t-4y^{3}-\frac{y_{t}}{\lambda-y}+\frac{3}{4}\frac{1}{(\lambda-y)^2}, \qquad t\neq p
\]
with \( p \) is a pole of the PI solution.

Under the transformation~\eqref{eq-def-hat{Phi}}, the asymptotic approximation of \( \Phi(\lambda, t) \) given in~\eqref{eq-canonical-solutions} becomes
\begin{equation}\label{eq-canonical-solutions-hat-Phi}
\widehat{\Phi}_{k}(\lambda,t)
=\frac{\lambda^{-\frac{3}{4}\sigma_{3}}}{\sqrt{2}}
\begin{bmatrix}1&-1\\1&1\end{bmatrix}
\left(I+\mathcal{O}(\lambda^{-\frac{1}{2}})\right)
e^{(\frac{4}{5}\lambda^{\frac{5}{2}}+t\lambda^{\frac{1}{2}})\sigma_{3}}
\end{equation}
as \( \lambda \to \infty \) with \( \lambda \in \Omega_{k} \), for \( t \) bounded away from the poles of the PI solution.


\section*{Acknowledgements}
The work of Wen-Gao Long was partially supported by the National Natural Science Foundation of China [Grant No. 12401094], the Natural Science Foundation of Hunan Province [Grant No. 2024JJ5131] and the Outstanding Youth Fund of Hunan Provincial Department of Education [Grant No. 23B0454]. The work of Yu-Tian Li was supported in part by the National Natural Science Foundation of China
[Grant no. 12071394] and the Shenzhen Science and Technology Program [Grant No. 20220816165920001].

\end{document}